\documentclass[12pt]{amsart}

\usepackage{amsmath}
\usepackage{amsfonts}
\usepackage{amscd}
\usepackage{amssymb}
\usepackage{amsthm}
\usepackage{mathdots}

\usepackage{MnSymbol}
\usepackage[linktocpage=true]{hyperref}
\usepackage{mathrsfs}
\usepackage{tikz-cd}

\usepackage{stmaryrd}
\usepackage{enumerate}
\usepackage{xcolor}

\hypersetup{
    bookmarks=true,         
    unicode=false,          
    pdftoolbar=true,        
    pdfmenubar=true,        
    pdffitwindow=false,     
    pdfstartview={FitH},    
    pdftitle={Moduli spaces of stable pairs},    
    pdfauthor={Yinbang Lin},     
    pdfsubject={},   
    pdfcreator={Yinbang Lin},   
    pdfproducer={}, 
    pdfkeywords={moduli space, stable pair, deformation-obstruction, virtual class}, 
    pdfnewwindow=true,      
    colorlinks=false,       
    linkcolor=black,          
    citecolor=black,        
    filecolor=magenta,      
    urlcolor=blue           
}

\makeatletter
\def\@tocline#1#2#3#4#5#6#7{\relax
  \ifnum #1>\c@tocdepth 
  \else
    \par \addpenalty\@secpenalty\addvspace{#2}%
    \begingroup \hyphenpenalty\@M
    \@ifempty{#4}{%
      \@tempdima\csname r@tocindent\number#1\endcsname\relax
    }{%
      \@tempdima#4\relax
    }%
    \parindent\z@ \leftskip#3\relax \advance\leftskip\@tempdima\relax
    \rightskip\@pnumwidth plus4em \parfillskip-\@pnumwidth
    #5\leavevmode\hskip-\@tempdima
      \ifcase #1
       \or\or \hskip 1em \or \hskip 2em \else \hskip 3em \fi%
      #6\nobreak\relax
   \hfill \hbox to\@pnumwidth{\@tocpagenum{#7}}\par
    \nobreak
    \endgroup
  \fi}
\makeatother

\usepackage{cite}

\newtheorem{thm}{Theorem}
\newtheorem*{thm*}{Theorem}
\newtheorem{lem}{Lemma}
\newtheorem{prop}{Proposition}
\newtheorem{cor}{Corollary}
\newtheorem{thm&defn}[thm]{Theorem \& Definition}
\newtheorem{defn}{Definition}

\newtheorem{innercustomthm}{Theorem}
\newenvironment{customthm}[1]
  {\renewcommand\theinnercustomthm{#1}\innercustomthm}
  {\endinnercustomthm}
  
\theoremstyle{remark}

\newcommand{\Z}{\mathbb Z}
\newcommand{\N}{\mathbb N}

\newcommand{\A}{\mathbb A}
\newcommand{\p}{\mathbb P}
\newcommand{\oo}{\mathscr O}

\newcommand{\spec}{{\rm Spec\,}}

\newcommand{\SL}{{\rm SL}}

\newcommand{\ob}{{\rm ob}}

\newcommand{\gr}{{\rm gr}}
\newcommand{\grass}{{\rm Grass}}
\newcommand{\quot}{{\rm Quot}}

\newcommand{\Id}{{\rm id}}
\newcommand{\id}{{\rm id}}

\newcommand{\im}{{\rm im\,}}
\newcommand{\coker}{{\rm coker\,}}

\newcommand{\Hom}{{\rm Hom}}
\newcommand{\HOM}{\mathcal Hom}
\newcommand{\ext}{{\rm ext}}
\newcommand{\Ext}{{\rm Ext}}

\newcommand{\pprime}{\prime\prime}

\author{Yinbang Lin}

\begin{document}
\title{Moduli spaces of stable pairs}



\maketitle

\begin{abstract}
We construct a moduli space of stable pairs over a smooth projective variety, parametrizing morphisms from 
a fixed coherent sheaf to a varying sheaf of fixed topological type, subject to a stability condition. 
This generalizes the notion used by Pandharipande and Thomas, following Le Potier,
where the fixed sheaf is the structure sheaf of the variety.  We then describe the relevant 
deformation and obstruction theories. We also show the existence of the virtual fundamental class in special cases.
\end{abstract}

\tableofcontents

\section{Introduction}
\vskip15pt
The past couple of decades of research have highlighted the importance of moduli spaces of {\it decorated} sheaves, which are sheaves with additional structure, such as one or more sections. Moduli spaces of rank two vector bundles with a section on a Riemann surface $X$, 
\[E\to X \quad\mbox{and}\quad \alpha:\oo_{X}\to E\] 
 were used in \cite{MR1273268} to deduce an important invariant of the moduli space of sheaves, the Verlinde number. More recently, Pandharipande and Thomas \cite{pandharipande-thomasi, pandharipande-thomasiii} 
studied stable pairs $(E,\alpha)$, where $E$ is a sheaf of dimension one, on a Calabi-Yau threefold. They showed that invariants of this moduli space are closely related to the Gromov-Witten invariants of the Calabi-Yau threefold. 

We would like to broaden our perspective and replace the structure sheaf by a general coherent sheaf. Subject to a stability condition, we would like to parametrize morphisms of coherent sheaves,
\[\alpha:E_0\to E,\]
where $E_0$ is a fixed coherent sheaf. We will denote such a morphism as a pair
\begin{equation*}(E,\alpha).\end{equation*}

Let us set up the problem. We will work over an algebraically closed field $k$ of characteristic $0$. We denote by $X$ a smooth projective variety of dimension $n$, with a fixed polarization $\oo_{X}(1)$. We fix a coherent sheaf $E_{0}$ on $X$. 
Let $P$ be a fixed polynomial of degree $d\leq n$. 
Let $\delta\in \mathbb Q[m]$ be another polynomial with a nonnegative leading coefficient; this will play the role of parameter for stability conditions. 

When $\delta$ is large, i.e. $\deg\delta\geq \deg P$, a pair $(E,\alpha)$, such that the Hilbert polynomial of $E$ equals $P$, is {\it stable} if $E$ is pure and the support of $\coker\alpha$ has dimension strictly smaller than $d$. 
This is the most significant case geometrically.
In this case, the moduli space of stable pairs is similar to Grothendieck's Quot scheme. 
But intersection theory on the moduli space of stable pairs is expected to be more tractable than that on the Quot scheme. 
This is because we impose the purity condition on the sheaves underlying stable pairs, which allows us to avoid some large dimensional components.

The moduli space of stable pairs in the large $\delta$ case is expected to have interesting applications to the enumerative geometry of higher rank sheaves on a surface $X$. 
In particular, a potential application is towards the strange duality conjecture. The conjecture over curves was proved \cite{MR2350055, MR2289865} by studying intersection theory on related Grassmannians and Quot schemes. It is reasonable to expect that a similar method using the moduli space of stable pairs will work for the surface case.

The study of stable pairs by Pandharipande and Thomas was built on Le Potier's work \cite{LePotier93} on coherent systems. The moduli space of coherent systems was also used to study the Donaldson numbers of the moduli space of sheaves \cite{MR1644040}. A {\it coherent system} on $X$ is a pair $(\Gamma, E)$, where  $E$ is a coherent sheaf and $\Gamma\leq H^{0}(X,E)$ is a subspace of global sections. 
 A pair $(E,\alpha: \oo_{X}\to E)$ can be viewed as a coherent system $(k\langle \alpha\rangle, E)$. However, when $\oo_{X}$ is replaced by, for example, $\oo_{X}^{\oplus 2}$, the pair can no longer be viewed as a coherent system, because the map 
 \[H^{0}(\alpha):k^{\oplus2}\to H^{0}(E)\]
  may not be injective. Aside from this issue, there is yet another difference between pairs and coherent systems: while the morphism $\alpha$ is part of the data of the pair, the coherent system only remembers the image of $H^{0}(\alpha)$. Consequently, when one tries to parametrize $\alpha:E_{0}\to E$ for general $E_{0}$, Le Potier's construction does not automatically apply.
 But the main ingredients of constructing the moduli space remain the same: Grothendieck's Quot scheme~\cite{MR1611822} and Mumford's geometric invariant theory~\cite{MR1304906}.
 
 We have
 \newpage

\begin{thm}[Existence of Moduli Spaces]
\label{MainThm}For the moduli functor $\mathcal S_{X}(P,\delta)$ of S-equivalence classes of $\delta$-semistable pairs, there exists a projective coarse moduli space $S_X(P,\delta)$. The moduli functor $\mathcal S_X^s(P,\delta)$ of equivalence classes of $\delta$-stable pairs is represented by an open subscheme $S_X^s(P,\delta)$ of $S_X(P,\delta)$.\end{thm}

Deformation-obstruction theory of stable pairs is very similar to that of the Quot scheme.
For a quotient $q:E_{0}\twoheadrightarrow F$, let $G=\ker q$, then we have a short exact sequence
\[0\to G\to E_{0}\to F\to 0\]
The deformation space, respectively the obstruction space, is
\begin{equation*}\Hom(G,F),\quad\mbox{respectively}\quad \Ext^{1}(G,F).\end{equation*}
Notice that $G$ is quasi-isomorphic to the cochain complex $J^{\bullet}=\{E_{0}\to F\}$, the deformation space, respectively the obstruction space, of this quotient is isomorphic to
\[\Hom(J^{\bullet},F),\quad\mbox{respectively}\quad \Ext^{1}(J^{\bullet},F).
\] 

The deformation-obstruction problem of stable pairs has a similar answer.
Let $\mathcal Art_{k}$ be the category of local Artinian $k$-algebras with residue field $k$. Let $A,B\in {\rm Ob\,}\mathcal Art_{k}$ and 
\[0\to K\to B\stackrel{\sigma}{\to} A\to 0\]
be a {\em small} extension, i.e. $\frak m_{B}K=0$. Suppose $(E,\alpha)$ is a stable pair. Let $I^{\bullet}$ denote the following cochain complex concentrating at degree $0$ and $1$,
\[I^{\bullet} =\{E_{0}\stackrel{\alpha}{\to} E\}.\]
\begin{thm}[Deformation-Obstruction]
\label{deformation-obstruction}
Suppose we have a morphism $\alpha_{A}:E_0\otimes_k A\to E_A$ over $X_A=X\times_{\spec k}\spec A$ extending $\alpha$, where $E_A$ is a coherent sheaf flat over $A$. There is a class 
\[\ob(\alpha_{A},\sigma)\in \Ext^{1}(I^{\bullet}, E\otimes K),\]
 such that there exists an extension of $\alpha_{A}$ over $X_{B}$ if and only if $\ob(\alpha_{A},\sigma)=0$. If extensions exist, the space of extensions is a torsor under 
 \[\Hom(I^{\bullet}, E\otimes K).\]
\end{thm}

In some special cases of moduli spaces of stable pairs, $\Ext^{i}(I^{\bullet},E)\not= 0$ only when $i=0,1$. In these cases, we will demonstrate the existence of the virtual fundamental class, which is important for the study of intersection theory on the moduli space.
\begin{thm}[Virtual Fundamental Class]\label{virtual-fundamental-class}
Suppose $X$ is a surface, $E_{0}$ is torsion free, $\deg P=1$, and $\deg\delta\geq1$. Then the moduli space $S_{X}^{s}(P,\delta)=S_{X}(P,\delta)$ of stable pairs admits a virtual fundamental class.\end{thm}

The virtual fundamental class can be used to define invariants of the surface. Kool and Thomas~\cite{kool-thomas-i,kool-thomas-ii} studied stable pairs invariants with $E_{0}\cong \oo_{X}$ on surfaces, using the reduced obstruction theory, which is necessary. We will address the intersection theory of the moduli space of stable pairs on a surface in future work.

\vskip6pt
After this paper was completed, I learned about the article~\cite{MR3360753},
where the stability condition for pairs had been defined and the small $\delta$ case of Theorem~\ref{MainThm} of this paper had been stated as the main theorem~\cite[Theorem 3.8]{MR3360753}.  
 In the large $\delta$ case, $\deg\delta\geq \deg P$, the linearized ample line bundle needs to be chosen differently~(\ref{linearized-line-bundle}) for the GIT construction. 
In this paper, a separate construction is carried out from a basic level. For example, Lemma~\ref{EquiSS2} is shown for characterizing stability in terms of global sections instead of Hilbert polynomials. 
 Theorem~\ref{MainThm} covers all cases of the construction, including the geometrically important large $\delta$ case. The large $\delta$ case is presented in the body of the paper in Sections~\ref{boundedness} and \ref{construction}. The small $\delta$ case where the construction was previously carried out by Wandel is included in the appendix for completeness.
  In view of Wandel's result, the contents of the appendix are not new.
 The body of the paper also contains in Section~\ref{deformation-theory} the deformation-obstruction theory, captured by Theorem~\ref{deformation-obstruction}. 
 Section~\ref{virtual-class} shows the existence of the virtual fundamental class in special geometries, Theorem~\ref{virtual-fundamental-class}.
 Section~\ref{examples} gives examples of smooth moduli spaces and calculate their topological Euler characteristics.
 Section~\ref{preliminary} sets the stage with preliminary results.
 
 I recently learned that the stable pair moduli space for $\deg \delta \geq \deg P$ was also previously studied in  \cite{hulls-husks}, where it appears as the moduli space of quotient husks. The author constructed it as a bounded proper separated algebraic space. The space was used in \cite{hulls-husks} to study an analogue of the flattening decomposition theorem for reflexive hulls. The current paper settles affirmatively the question raised in \cite{hulls-husks} regarding the projectivity of the space.

I finally note that once the moduli space is constructed for $\deg\delta< \deg P$, it is available in an indirect way for $\deg\delta\geq \deg P$ as well. This follows from the finiteness of the set of critical values and the fact that the largest critical polynomial $\delta_{\max}$ has $\deg\delta_{\max}< \deg P$. Then the stability polynomial $\delta^{\prime}$ can be taken to be of degree $\deg P-1$ and larger than $\delta_{\max}$. For any $\delta$ with $\deg\delta\geq \deg P$, we have $S_{X}(P,\delta)\cong S_{X}(P,\delta^{\prime})$. Although this observation is not made in \cite{MR3360753}, the author proves the set of critical $\delta$'s is finite. This is also included in the current appendix with a different proof.

This indirect argument does not however yield the linearized ample line bundle for $S_{X}(P,\delta)$ with $\deg\delta\geq \deg P$. For stability polynomials $\delta^{\prime}$ with $\deg\delta ^{\prime}<\deg P$, the linearization depends directly on $\delta^{\prime}$; the highest critical polynomial $\delta _{\max}$ cannot be determined explicitly though, since the boundedness which underlies the finiteness of the set of critical stability values is itself not explicit. 
 
For some applications, it is nevertheless important to know the line bundle explicitly. A natural problem to study next is that of wall-crossing formulas, using Thaddeus' master space~\cite{MR1333296,mochizuki}. The construction of the master space requires the linearized ample line bundle. So, it is important to construct the moduli space directly via GIT and obtain the ample line bundle. I will address the problem of wall-crossing formulas in future work.

\vskip8pt
 {\bf Acknowledgements.} The author is very grateful to Professor Alina Marian for introducing him to this problem and discussions throughout the process of studying this problem and preparing this paper. The author also wants to thank Barbara Bolognese, Prof. Daniel Huybrechts, Prof. Anthony Iarrobino, Yaping Yang, and Gufang Zhao for the discussions and correspondences.
\vskip45pt

\section{Basic properties of stable pairs}\label{preliminary}

\vskip35pt

\subsection{Preliminaries on coherent sheaves}

For a coherent sheaf $E$ on $(X,\oo_{X}(1))$, we denote by $P_{E}$ its {\it Hilbert polynomial}. Recall that, we can write the Hilbert polynomial in the following form
\begin{equation*}
P_{E}(m)=\sum^{d}_{i=0}a_{i}(E)\frac{m^{i}}{i!},\end{equation*}
where $d=\dim E$ is the dimension of the support of $E$ and $a_i(E)\in \Z$. We denote by 
\[r(E)=a_{d}(E)\]
the {\it multiplicity} of $E$. Let 
\[p_{E}=\frac{P_{E}}{r(E)}\]
be the {\it reduced} Hilbert polynomial. For a coherent sheaf $E$, the {\it slope} of $E$ is 
\[\mu(E)=\frac{a_{d-1}(E)}{a_{d}(E)}.\]
Polynomials are ordered lexicographically.

A coherent sheaf $E$ is {\it pure} if there is no subsheaf of lower dimensional support. 
It is {\it semistable} (respectively slope-semistable), if it is pure and there is no subsheaf with larger reduced Hilbert polynomial (respectively slope). 
For a pure sheaf, there is a {\it Harder-Narasimhan filtration} with respect to slope
\[0\subsetneqq E_{1}\subsetneqq E_{2}\subsetneqq\cdots \subsetneqq E_{l}=E,\]
where $E_{t+1}/E_{t}$ is slope semistable and 
\[\mu(E_{t}/E_{t-1})>\mu(E_{t+1}/E_{t}),\quad \forall t\in [1,l-1].\]
We shall denote 
\[\mu_{\max}(E)=\mu(E_{1})\quad \mbox{and} \quad\mu_{\min}(E)=\mu(E_{l}/E_{l-1}).\]

To construct the moduli space via GIT, the first step is to prove a boundedness result. For our convenience, we group a sequence of boundedness results here.
\begin{customthm}{G}[Grothendieck]\label{GrSlope}Suppose $F$ is a pure coherent $\oo_{X}$-module of dimension $d$. Then: \begin{enumerate}[(i)]
\item the slopes of nonzero coherent subsheaves are bounded above;
\item the family of subsheaves $F^{\prime}\subset F$ with slopes bounded below, such that the quotient $F/F'$ is pure and of dimension $d$, is bounded.
\end{enumerate}
\end{customthm}
We can also make a statement similar to the second assertion about the boundedness of quotients. For the proof of this basic theorem, see \cite[Lemma 2.5]{MR1611822}.

Let $Y$ be the scheme theoretic support of a pure sheaf $E$ of dimension $d$ and multiplicity $r=r(E)$. We include the following results discussed in
\cite{LePotier93}.
\begin{lem}The degree of $Y$ is no larger than $r^{2}$.\end{lem}
\begin{proof}This is clear from an equivalent definition of multiplicity \cite[Definition 2.1]{LePotier93}.
\end{proof}

\begin{lem}\label{MuMin}The minimum slope $\mu_{\min}(\oo_{Y})$ is bounded below by a constant determined by $n$, 
$r$ and $d$.\end{lem}
\begin{proof}See \cite[Lemma 2.12]{LePotier93}.\end{proof}

The following statement \cite[Theorem 1.1]{SimpsonI} is crucial to our proof of boundedness.
\begin{customthm}{S}[Simpson]\label{BddMuMax}Let $C$ be a rational 
 constant. The family of pure coherent sheaves $E$ with Hilbert polynomial $P_{E}=P$, such that 
\[\mu_{\max}(E)\leq C,\] is bounded.\end{customthm}

Bounding $\mu_{\max}$ from above is equivalent to bounding $\mu_{\min}$ from below, because Hilbert polynomial is additive in a short exact sequence.

We will also need the following statement~\cite[Corollary 1.7]{SimpsonI}.
\begin{lem}[Simpson]\label{SimpsonH0}Suppose $F$ is a slope semistable sheaf of dimension $d$, multiplicity $r$ and slope $\mu$. There is a constant $C$ depending on $r$ and $d$ 
 such that\footnote{$[x]_+=\max\{0,x\}.$}
\[\frac{h^{0}(F)}{r}\leq\frac{1}{d!}([\mu+C]_{+})^{d}.\]\end{lem}
\vskip20pt
\subsection{Stable pairs}

Let $E_0$ be a coherent sheaf on $X$. Let $P$ be a polynomial of degree $d$, and $\delta$ a polynomial with a nonnegative leading coefficient.

\begin{defn}A {\rm pair} 
\[(E, \alpha)\]
(of type $P$) is a morphism $\alpha:E_{0}\to E$
 of coherent sheaves on $X$, where $P_{E}=P$.
 A {\rm sub-pair} 
 \[(E^{\prime},\alpha^{\prime})\subset(E,\alpha),\]
  is a morphism $\alpha^{\prime}:E_{0}\to E'$, such that $E' \subset E$ and 
  \[\Big\{\begin{tabular}{ll}
$\iota \circ \alpha^{\prime}= \alpha$ & if $E'\supset\im\alpha$,\\
$\alpha^{\prime}=0$ &  otherwise.
\end{tabular}\]
Here, $\iota$ denotes the inclusion $E^{\prime}\hookrightarrow E$. A {\rm quotient pair} 
\[(E'',\alpha^{\pprime})\]
 is a coherent quotient sheaf $q:E\to E''$ with 
\[\alpha^{\pprime}=q\circ \alpha:E_{0}\to E''.\]
\end{defn}
We say a pair $(E, \alpha)$ has dimension $d$ if $\dim E=d$.

A {\it morphism} $\phi:(E,\alpha)\to(F,\beta)$ of pairs is a morphism of sheaves $\phi:E\to F$ such that there is a constant $b\in k$
\[\phi\circ\alpha= b\beta.\] 
By this definition, sub-pairs and quotient pairs can be viewed as morphisms. For simplicity, we shall use the notation $\phi$ for both the morphism of pairs and that of their underlying sheaves.

A {\it short exact sequence of pairs}
\[0\to (E',\alpha^{\prime})\stackrel{\iota}{\to} (E,\alpha)\stackrel{q}{\to} (E'',\alpha^{\pprime})\to0\]
consists of a short exact sequence of sheaves $0\to E^{\prime}\to E\to E''\to 0$, 
such that $(E',\alpha^{\prime})$ is a sub-pair and $(E'',\alpha^{\pprime})$ the corresponding quotient pair. More precisely,
\[\alpha^{\pprime}=q\circ \alpha \mbox{ if }\alpha^{\prime}=0,\quad \mbox{and}\quad \alpha^{\pprime}=0 \mbox{ if }\iota\circ\alpha^{\prime}=\alpha.\]

The {\it Hilbert polynomial} ({\it reduced Hilbert polynomial} resp.) of a pair $(E,\alpha)$ is 
\[P_{(E,\alpha)}=P_{E}+\epsilon(\alpha)\delta\quad (p_{(E,\alpha)}=p_{E}+\frac{\epsilon(\alpha)\delta}{r(E)} \mbox{ resp.}).\]
Here, 
\[\epsilon(\alpha)=\Big\{\begin{tabular}{ll}
1 & if $\alpha\not=0$,\\
0 & otherwise.
\end{tabular}\]
Clearly, Hilbert polynomials are additive in a short exact sequence of pairs.

\begin{defn}\label{Stability}A pair $(E,\alpha)$ is $\delta$-{\rm stable} if
\begin{enumerate}[(i)]
\item $E$ is pure;
\item for every proper sub-pair $(E',\alpha^{\prime})$, 
\[p_{(E',\alpha^{\prime})}< p_{(E,\alpha)}\] 
\end{enumerate}
Semistability is defined similarly, replacing the strong inequality by the corresponding weak inequality.
\end{defn}
The second condition is equivalent to that for every proper quotient pair $(E'',\alpha^{\pprime})$ of dimension $d$, 
\[p_{(E'',\alpha^{\pprime})}> p_{(E,\alpha)}.\]

\vskip6pt
{\bf Convention.} In the rest of this paper, if stability is characterized by a strong inequality, semistability can be characterized by the corresponding weak inequality. So, in such a case, we will only make the statement for stability.
\vskip6pt
When the context is clear, we will omit $\delta$ and only say a pair is stable or semistable.

Clearly, a pair $(E,0)$ is (semi-)stable if and only if $E$ is (semi-)stable as a coherent sheaf. We will call a pair $(E,\alpha)$ {\it non-degenerate} if $\alpha\not=0$. We are primarily interested in non-degenerate semistable pairs, which we are going to parametrize.

A family of pairs parametrized by a scheme $T$ is a morphism of sheaves
\[\alpha_{T}:\pi_{2}^{*}E_{0}\to \mathscr E\]
over $T\times X$, such that $\mathscr E$ is flat over $T$. Here, $\pi_{2}$ is the projection 
\[T\times X\rightarrow X.\]
Two families $\alpha_{T}:\pi_2^*E_0\to \mathscr{E}$, $\beta_{T}:\pi_2^*E_0\to \mathscr{F}$ are equivalent, if there is an isomorphism 
\[\psi:\mathscr{E}\to\mathscr {F},\quad\mbox{such that }\psi\circ\alpha_{T}=\beta_{T}.\]

In the large $\delta$ regime, semistable pairs have some special features:
\begin{lem}When $\deg \delta\geq \deg P$, there is no non-degenerate strictly semistable pair, that is, every non-degenerate semistable pair is stable.\end{lem}
\begin{proof}Suppose $(G,\alpha^{\prime})$ is a sub-pair of a semistable $(E,\alpha)$, such that $p_{(G,\alpha^{\prime})}=p_{(E,\alpha)}$, i.e. 
\[p_{G}+\frac{\epsilon(\alpha^{\prime})\delta}{r(G)}= p_{E}+\frac{\delta}{r( E)}.\] 
Consider the leading coefficients. Because  
$\deg \delta\geq d$, $\epsilon(\alpha^{\prime})=1$ and $r(G)= r(F)$. Thus, $p_{E}=p_{G}$. Therefore, $P_{E}=P_{G}$, which implies that $G=E$. Hence, $(G,\alpha^{\prime})=(E,\alpha)$. We have shown that $(E,\alpha)$ is not strictly semistable.\end{proof}
We also have a reinterpretation of the stability condition. 
\begin{lem}\label{ReIntSS}Suppose $E$ is a pure coherent sheaf with Hilbert polynomial $P_{E}=P$ and multiplicity $r(E)=r$. If $\deg\delta\geq d=\deg P$, then a pair $(E,\alpha)$ is stable if and only if for every proper sub-pair $(G,\alpha^{\prime})$,
\[\frac{P_{G}}{2r(G)-\epsilon(\alpha^{\prime})}<\frac{P}{2r-\epsilon(\alpha)}.\]\end{lem}
\begin{proof}When $\deg\delta\geq d$, for any proper sub-pair $(G,\alpha^{\prime})$, the inequality 
\[p_{G}+\epsilon(\alpha^{\prime})\frac{\delta}{r(G)}< p_{E}+\epsilon(\alpha)\frac{\delta}{r}\]
is equivalent to
\begin{equation}\label{InWords}
\frac{\epsilon(\alpha^{\prime})}{r(G)}\leq\frac{\epsilon(\alpha)}{r},\mbox{\quad and in case of equality, \quad}p_{G}< p_E.
\end{equation}
The latter can be easily seen to be equivalent to\begin{equation*}
\frac{r(G)}{2r(G)-\epsilon(\alpha^{\prime})}\leq\frac{r}{2r-\epsilon(\alpha)},\mbox{\quad and in case of equality, \quad}p_{G}< p_E.
\end{equation*}This last condition is equivalent to the inequality in the statement.
\end{proof}
Moreover, there is a geometric characterization of stability.
\begin{lem}\label{GenSurj}If $\deg \delta\geq \deg P$, then $(E,\alpha)$ is stable if and only if $E$ is pure and $\dim \coker\alpha<\deg P$.
\end{lem}
\begin{proof}Let $G$ be the image of $\alpha$. A priori, $r(G)\leq r(E)$. The stability implies that $r( G)\geq r( E)$, as in (\ref{InWords}). 
Thus, $r(G)=r(E)$, which implies that $E/G$ has a Hilbert polynomial of degree strictly less than $d$. Hence, $\dim \coker\alpha<d$.

To prove the other direction, suppose $\dim \coker\alpha<d$. For an arbitrary subsheaf $G\subset E$, if $\im \alpha\subset G$, then $r( G)= r( E)$. We also have $P_G\leq P_E$, thus \[p_{G}+\frac{\delta}{r( G)}\leq p_{E}+\frac{\delta}{r(E)}.\]
If $\im \alpha\nsubset G$, then 
\[p_G\leq p_{E}+\frac{\delta}{r(E)}\]
 Therefore, the converse is also true.\end{proof}
Pairs share some similar properties of sheaves.
\begin{lem}\label{SmallLarge}Suppose $\phi:(E,\alpha)\to(F,\beta)$ is a nonzero morphism of pairs.
\begin{enumerate}[(i)]
\item Suppose $(E,\alpha)$ and $(F,\beta)$ are $\delta$-semistable of dimension $d$. Then 
\[p_{(E,\alpha)}\leq p_{(F,\beta)}.\]
    \item If $(E,\alpha)$ and $(F,\beta)$ are $\delta$-stable with the same reduced Hilbert polynomial, then $\phi$ induces an isomorphism between $E$ and $F$. In particular, for a stable pair 
    $(E,\alpha)$, \[{\rm End}((E,\alpha))\cong k.\]
\end{enumerate}\end{lem}
\begin{proof}
(i) Let $\alpha^{\prime\prime}$ be $\phi\circ\alpha:E_{0}\to \im\phi$. Then $(\im\phi,\alpha^{\pprime})$ is a quotient pair of $(E,\alpha)$ and a sub-pair of $(F,\beta)$. Thus, 
\begin{equation}\label{nondecreasing-hilbert-poly}p_{(E,\alpha)}\leq p_{(\im\phi,\alpha^{\pprime})}\leq p_{(F,\beta)}.
\end{equation}

    (ii) Suppose not, then $\ker\phi\not=0$ or $\im\phi\not= E$. We also have the inequalities (\ref{nondecreasing-hilbert-poly}).
    But two equalities do not hold simultaneously, which contradicts the fact that the two stable pairs have the same reduced Hilbert polynomial. Therefore, $\ker\phi=0$ and $\im\phi= E$. Thus, $\phi$ is an isomorphism of coherent sheaves. Clearly, the inverse also provides an inverse of pairs. In particular, ${\rm End}((E,\alpha))$ is a finite dimensional associative division algebra over the algebraically closed field $k$, hence $k$.\end{proof}

\begin{prop}[Harder-Narasimhan Filtration]\label{HN} Let $(E,\alpha)$ be a pair where $E$ is pure of dimension $d$. Then there is a unique filtration by sub-pairs
\[0\subsetneqq (G_{1},\alpha_{1})\subsetneqq (G_2,\alpha_{2})\subsetneqq\cdots\subsetneqq (G_l,\alpha_{l})=(E,\alpha)\]
with 
\[\gr_i=
(G_i,\alpha_{i})/(G_{i-1},\alpha_{i-1})\] satisfying
\begin{enumerate}[(i)]
\item $\gr_i$ is $\delta$-semistable of dimension d for all $i$;
\item $p_{\gr_i}>p_{\gr_{i+1}}$, for all $i$.\end{enumerate}
We call this filtration the {\rm Harder-Narasimhan filtration} of the pair.\end{prop}
\begin{proof}Because Hilbert polynomials are additive in a short exact sequence of pairs, the proof is the same as the proof of the existence and uniqueness of the Harder-Narasimhan filtration of a pure sheaf \cite[Theorem 1]{MR0498573}.
\end{proof}

Evidently, in the filtration, there is only one nonzero $\alpha_i$. In the case where $\deg \delta\geq d$, only $\alpha_{1}$ is nonzero.

\begin{prop}[Jordan-H\"older Filtration] Let $(E,\alpha)$ be a semistable pair. There is a filtration
\[0\subsetneqq (F_{1},\alpha_{1})\subsetneqq (F_2,\alpha_{2})\subsetneqq\cdots\subsetneqq (F_l,\alpha_{l})=(E,\alpha),\]
such that each factor 
\[\gr_{i}=(F_i,\alpha_{i})/(F_{i-1},\alpha_{i-1})\]
 is stable with reduced Hilbert polynomial $p_{(E,\alpha)}$. Moreover, $\gr(E,\alpha)=\oplus \gr_{i}$ does not depend on the filtration.
\end{prop}
\begin{proof}Since we have Lemma~\ref{SmallLarge}, 
the proof goes the same as the argument for Jordan-H\"older filtrations of a semistable sheaf, see e.g. \cite[Proposition 1.5.2]{huybrechts-lehn-moduli-of-sheaves}.
\end{proof}
Two semistable pairs are {\it S-equivalent}, if they have isomorphic Jordan-H\"older factors.
Let 
\[\mathcal S_{X}(P,\delta):\mathcal Sch_{/k}\to \mathcal Set\]
denote the moduli functor of S-equivalent non-degenerate semistable pairs of type $P$.
Let 
\[\mathcal S^{s}_{X}(P,\delta)\]
 denote the moduli functor of equivalence classes of non-degenerate stable pairs.
\vskip45pt

\section{Boundedness when $\deg \delta\geq\deg P$}\label{boundedness}
\vskip35pt
In order to construct the moduli space via GIT, we first need to prove that the family of semistable pairs is bounded.

In this and the next section, many statements are true either $\deg \delta \geq \deg P$ or $\deg\delta <\deg P$, but require different proofs. In these two sections, we will only treat the case where 
\[\deg \delta \geq \deg P=d.\]
 In Appendix, we will point out modifications needed for the proofs in the cases where $\deg\delta <d$.
 
We will show boundedness using Theorem~\ref{BddMuMax}, by studying $\mu_{\min}$ of sheaves underlying semistable pairs.

\begin{lem}\label{mu-min}
Fix the Hilbert polynomial $P$. Suppose $(E,\alpha)$ be a pair, which is semistable for some $\delta$, with $P_{E}=P$. Then, $\mu_{\min}(E)$ is bounded below by a constant depending on $P$ and $X$.
\end{lem}
We would like to emphasize that the constant is independent of $\delta$. 
\begin{proof}
Let $(E,\alpha)$ be a semistable pair. By Lemma~\ref{GenSurj},
\begin{equation}\label{gensurj}\dim \coker\alpha<d.\end{equation}
 Choose an $m$ large enough such that 
 \[H^{0}(E_{0}(m))\otimes\oo_{X}(-m)\twoheadrightarrow E_{0}.\]
 Let $Y$ be the scheme theoretic support of $E$.
 The morphism $\alpha$ factors through $E_{0}|_{Y}$. 
 We have the following sequence of morphisms
 \[H^{0}(E_{0}(m))\otimes\oo_{Y}(-m)\twoheadrightarrow E_{0}|_{Y}\to E\twoheadrightarrow \gr_{s}E,\]
 where the last morphism is the surjection from $E$ onto its last factor of Harder-Narasimhan filtration with respect to slope. 
 By (\ref{gensurj}), the composition is nonzero. 
 Therefore,
 \[\mu_{\min}(E)=\mu(\gr_{s}E)\geq\mu_{\min}(H^{0}(E_{0}(m))\otimes\oo_{Y}(-m))=\mu_{\min}(\oo_{Y}(-m))=\mu_{\min}(\oo_{Y})-m,\]
 where the last term is bounded below by Lemma~\ref{MuMin}. 
 Thus, $\mu_{\min}(E)$ is bounded below by a constant, which depends on $X$ and $P$.
\end{proof}
Combining Lemma~\ref{mu-min} and Theorem~\ref{BddMuMax}, we obtain the following boundedness result.

\begin{prop}\label{Bdd}Fix the Hilbert polynomial $P$. The family 
\[\{E|(E,\alpha)\mbox{ is a semistable pair of type }P\ \mbox{w.r.t. some }\delta.\}\]
 of coherent sheaves on $X$ is bounded.
\end{prop}

Next, we shall prove that, for a bounded family of pure pairs, the family of factors of their Harder-Narasimhan filtrations is bounded:
\begin{lem}\label{HNGrBdd}Suppose $\Phi:\pi_2^*E_0\to\mathscr E$ over $T\times X$ is a flat family of pure pairs over $X$ parametrized by a finite type scheme $T$. 
For each closed point $t\in T$, let $\{(\gr_i^t,\alpha_i^t)\}_{i\in I_t}$ be the Harder-Narasimhan factors of $(\mathscr E(t),\Phi(t))$, where $\mathscr E(t)=\mathscr E|_{\spec k(t)\times X}$ and $\Phi(t)$ is the corresponding morphism. 
Then, the family $\{\gr_i^t\}_{t\in T,i\in I_t}$ is bounded.
\end{lem}
The following proof is very similar to the proof of the corresponding statement about the boundedness of Harder-Narasimhan factors of pure sheaves \cite[Theorem 2.3.2]{huybrechts-lehn-moduli-of-sheaves}.
The proof is independent of $\deg\delta$.
\begin{proof}
We can assume $T$ to be integral. 
Define $A$ as the set of $2$-tuples $(P'',\epsilon'')$, such that there is a $t\in T$ and a pure quotient $q:\mathscr E(t)\twoheadrightarrow E''$ with Hilbert polynomial $P_{E^{\prime\prime}}=P^{\prime\prime}$ and $\epsilon''=\epsilon(q\circ\Phi(t))$, which destabilizes $(\mathscr E(t),\Phi(t))$:
\[p''+\frac{\epsilon''\delta}{r''}<p+\frac{\epsilon(\Phi(s))\delta}{r}.\]
Here, $p$ and $p''$ denote the corresponding reduced Hilbert polynomials, $r$ and $r''$ denote the multiplicities. 
From this inequality, we know
 that $\mu(E'')$ is bounded above by a constant determined by $P$ and $\delta$. Therefore, $A$ is a finite set by Theorem~\ref{GrSlope}.
 
If this set is empty, then all pairs are semistable. Then, we are done. Otherwise, let's consider whether there is a $(P_{-},\epsilon_{-})$, which is minimal with respect to the total order $\preceq$ and satisfies the condition that for a generic point $t\in T$, there is a pure quotient $q:\mathscr E(t)\to F$ with 
\begin{equation}\label{minimum-reduced-hilbert-polynomial}P_{F}=P_{-}\quad \mbox{and} \quad\epsilon(q\circ \Phi(t))=\epsilon_{-}.\end{equation}
 The order $\preceq$ is defined as follows: $(P_{1},\epsilon_{1})\preceq(P_{2},\epsilon_{2})$
 \[\mbox{if }\ p_{1}+\epsilon_{1}\delta/r_{1}\leq p_{2}+\epsilon_{2}\delta/r_{2},\mbox{ and in the case of }=, \,\,P_{1}\geq P_{2}.\]
 This is to pick out the maximal semistable quotient pair with the minimum reduced Hilbert polynomial.
 \vskip6pt
(i) If there is no such a $(P_{-},\epsilon_{-})$, then generically, say over the open subscheme $U\subset T$, pairs are already semistable. 
 
(ii) If there is such a $(P_{-},\epsilon_{-})$, let $U\subset T$ be the open family having quotients satisfying the condition (\ref{minimum-reduced-hilbert-polynomial}). The minimal Harder-Narasimhan factors of pairs in $U$ are parametrized by a subscheme of $\quot^{P_{-}}(\mathscr E)$. To parametrize all the Harder-Narasimhan factors of pairs parametrized by $U$, we can iterate the above process for the kernel, which is flat, of the universal quotient over $\quot^{P_{-}}(\mathscr E)$. This process will terminate due the multiplicity reason.
\vskip6pt

Then, we can run the same algorithm for pairs parametrized by irreducible components of the complement $T\setminus U$. Because $T$ is notherian, the process will terminate.

We have thus parametrized the Harder-Narasimhan factors by a finite sequence of Quot schemes.
\end{proof}

The following statement enables us to handle the semistability condition via spaces of global sections, instead of Hilbert polynomials.

\begin{lem}\label{EquiSS2}Fix $P$ and $\delta$ with $\deg\delta\geq\deg P$. Then there is an $m_{0}\in \Z_{>0}$, such that for any integer $m\geq m_{0}$ and any pair $(E,\alpha)$, where $E$ is pure with $P_{E}=P$ and multiplicity $r(E)=r$, the following assertions are equivalent.
\begin{enumerate}[i)]
\item The pair $(E,\alpha)$ is stable.
\item $P_{E}(m)\leq h^{0}(E(m))$, and for any proper sub-pair $(G,\alpha^{\prime})$ with $G$ of multiplicity $r(G)$,
\[\frac{h^{0}(G(m))}{2r(G)-\epsilon(\alpha^{\prime})}< \frac{h^{0}(E(m))}{2r-\epsilon(\alpha)}.\]
\item For any proper quotient pair $(F,\alpha^{\prime\prime})$ with $F$ of dimension $d$ and multiplicity $r(F)$,
\[\frac{h^{0}(F(m))}{2r(F)-\epsilon(\alpha^{\pprime})}>  \frac{P(m)}{2r-\epsilon(\alpha)}.\]
\end{enumerate}
\end{lem}
The proof is modified from that of a similar statement in \cite{LePotier93}.
\begin{proof} The proof will proceed as follows: $i)\Rightarrow ii)\Rightarrow iii)\Rightarrow i)$. The integer $m_{0}$ will be determined in the course of the proof, non-explicitly.
\vskip6pt
i) $\Rightarrow$ ii): The family of sheaves underlying semistable pairs with a fixed Hilbert polynomial is bounded. Thus, there is $m_{0}\in \N$
 such that for any integer $m\geq m_{0}$, $H^{i}(E(m))=0$, $\forall i>0$. In particular, 
\[P(m)=h^{0}(E(m)).\]
 In the course of proving the boundedness, we also prove that $\mu_{\max}(E)$ is bounded above, say
\[\mu_{\max}(E)\leq \mu.\] 
For a proper sub-pair $(G,\alpha^{\prime})$ of multiplicity $r(G)$, consider the Harder-Narasimhan filtration of $G$ with respect to slope. 
Let 
\[\nu=\mu_{\min}(G).\] 
By Lemma~\ref{SimpsonH0}, we can find a constant $B$ depending on $r$ and $d$, 
such that
\begin{eqnarray}\label{1to2BySimpson}
\frac{h^{0}(G(m))}{r(G)}
\leq\frac{1}{d!}\big((1-\frac{1}{r})([\mu+m+B]_{+})^{d}+\frac{1}{r}([\nu+m+B]_{+})^{d}\big)
\end{eqnarray}
Choose a constant $A>0$,
 which is larger than all roots of $P$. Replace $m_{0}$ by $\max\{m_{0},A\}$. Then \[h^{0}(E(m))=P(m)\geq\frac{r}{d!}(m-A)^{d},\ \forall m\geq m_{0}.\]
Suppose $\nu_{0}$ is an integer such that 
\[B+\mu(1-\frac{1}{r})+\frac{\nu_{0}}{r}<-A.\]

 Enlarge $m_{0}$ if necessary, we have 
\begin{equation}\label{1to2}
\frac{1}{d!}\big((1-\frac{1}{r})([\mu+m+B]_{+})^{d}+\frac{1}{r}([\nu_{0}+m+B]_{+})^{d}\big)<\frac{P(m)}{r},\ 
\forall m\geq m_{0},\end{equation}
by considering the first and the second leading coefficients. 
Thus, when $m\geq m_{0}$ and $\nu\leq\nu_{0}$, combining (\ref{1to2BySimpson}) and (\ref{1to2}), we get
 \begin{equation}\label{nu<nu0}h^{0}(G(m))<\frac{r(G)}{r}h^{0}(E(m))\leq\frac{2r(G)-\epsilon(\alpha^{\prime})}{2r-\epsilon(\alpha)}h^{0}(E(m)).\end{equation}
  The last weak inequality is a consequence of (\ref{InWords}).

We are left to consider the case where $\nu>\nu_{0}$. First, notice that we can assume $E/G$ to be pure. If not, consider the saturation of $G$ in $E$, namely, the smallest $\bar G\supset G$, such that $E/\bar G$ is pure. If we can prove the inequality in ii) for $\bar G$, then it's also true for $G$, since 
\[r(G)=r(\bar G)\quad \mbox{and} \quad h^{0}(G(m))\leq h^{0}(\bar G(m)).\]
Notice that $\mu(G)\geq\nu>\nu_{0}$, the family of such $G$ is bounded, by Theorem~\ref{GrSlope}. So, there are only finitely many Hilbert polynomials of the form $P_{G}$ for such $G$. Moreover, we can enlarge $m_{0}$ again, if necessary, such that for $m\geq m_{0}$, 
\[P_{G}(m)=h^{0}(G(m))\quad \mbox{and}\]
\[\frac{P_{G}}{2r(G)-\epsilon(\alpha^{\prime})}< \frac{P}{2r-\epsilon(\alpha)} \iff \frac{P_{G}(m)}{2r(G)-\epsilon(\alpha^{\prime})}< \frac{P(m)}{2r-\epsilon(\alpha)}.\]
Therefore, by Lemma~\ref{ReIntSS} and (\ref{nu<nu0}), 
\[\frac{h^{0}(G(m))}{2r(G)-\epsilon(\alpha^{\prime})}< \frac{h^{0}(E(m))}{2r-\epsilon(\alpha)}.\]
\vskip6pt
ii) $ \Rightarrow$ iii): From a proper quotient pair $(F,\alpha^{\pprime})$, we can get a short exact sequence
\[0\to(G,\alpha^{\prime})\to(E,\alpha)\to(F,\alpha^{\pprime})\to 0.\]We thus obtain an exact sequence
\begin{equation}\label{2to3}0\to H^{0}(G(m))\to H^{0}(E(m))\to H^{0}(F(m)).\end{equation}
Therefore, 
\[h^{0}(F(m))\geq h^{0}(E(m))-h^{0}(G(m)).\]
 Notice that 
 \[r(E)=r(G)+r(F)\quad \mbox{and} \quad \epsilon(\alpha)=\epsilon(\alpha^{\prime})+\epsilon(\alpha^{\pprime}).\]
Thus, 
\[\frac{h^{0}(F(m))}{2r(F)-\epsilon(\alpha^{\pprime})}\geq \frac{h^{0}(E(m))-h^{0}(G(m))}{(2r-\epsilon(\alpha))-(2r(G)-\epsilon(\alpha^{\prime}))}> \frac{h^{0}(E(m))}{2r-\epsilon(\alpha)}\geq\frac{P(m)}{2r-\epsilon(\alpha)}.\]

\vskip6pt
iii) $\Rightarrow$ i): Take the Harder-Narasimhan filtration of $E$ with respect to slope. Suppose $F$ is the last factor, then $\mu(F)=\mu_{\min}(E)$, denoted as $\mu''$. 
By Lemma~\ref{SimpsonH0},
\begin{equation}\label{3to1}
 \frac{h^0(F(m))}{r(F)}\leq\frac{1}{d!}([\mu''+m+C]_{+})^{d}.\end{equation}
Let $(F,\alpha^{\pprime})$ be the induced quotient pair. If $\epsilon(\alpha^{\pprime})\not=0$, then $(E,\alpha)$ is stable, since in the Harder-Narasimhan filtration, only the first morphism is nonzero. So, assume 
 $\epsilon(\alpha^{\pprime})=0$. Then
\[\frac{P(m)}{r}<\frac{2P(m)}{2r-\epsilon(\alpha)}<\frac{h^{0}(F(m))}{2r(F)}\leq\frac{1}{d!}([\mu''+m+C]_{+})^{d}.\]
 If $m\geq m_{0}$, the preceding inequality with $P(m)/r \geq (m-A)^{d}/d!$ implies that $m-A\leq \mu''+m+C$. Therefore, 
 \[\mu_{\min}(E)=\mu''\geq-A-C.\]
   Thus, 
   the family of coherent sheaves satisfying the third condition for an $m\geq m_{0}$ is bounded.

Let $\gr_{s}=(\gr_{s}E,\gr_{s}\alpha)$ denote the last Harder-Narasimhan factor of the pair $(E,\alpha)$.
Then
\[\frac{h^{0}(\gr_{s}E(m))}{2r(\gr_{s}E)-\epsilon(\gr_{s}\alpha)}>  \frac{P(m)}{2r-1}.\]
By Lemma~\ref{HNGrBdd}, enlarge $m_{0}$ if necessary, we can assume that, $\forall m\geq m_{0}$,
\begin{enumerate}[(i)]
\item $h^{0}(\gr_{s}E(m))=P_{\gr_{s}E}(m)$;
\item \[\frac {P_{\gr_{s}E}(m)}{2r(\gr_{s}E)-\epsilon(\gr_{s}\alpha)}> \frac{P(m)}{2r-1}\iff \frac {P_{\gr_{s}}}{2r(\gr_{s}E)-\epsilon(\gr_{s}\alpha)}>  \frac{P}{2r-1}.\]
\end{enumerate}
Therefore, 
\[\frac{\epsilon(\gr_{i}\alpha)}{r(\gr_{s}E)}\geq \frac{1}{r},\] 
which implies $\epsilon(\gr_{s}\alpha)=1$. Thus, $s=1$, which means $(E,\alpha)$ is semistable, thus stable.
\end{proof}
Replacing the strong inequalities by weak inequalities, the lemma is also true.
\vskip45pt

\section{Construction of the moduli space when $\deg \delta\geq\deg P$}\label{construction}
\vskip35pt

Fix the smooth projective variety $(X,\oo_{X}(1))$, the coherent sheaf $E_{0}$, the Hilbert polynomial $P$, and the stability condition $\delta$.

By boundedness results proven in the last section,  
there is an $N\in \Z$ such that for any integer \[m>N,\]
 the following conditions are satisfied:\begin{enumerate}[(i)]
\item $E_{0}(m)$ is globally generated. 
\item $E(m)$ is globally generated and has no higher cohomology, for every $E$ appearing in a $\delta$-semistable pair (Proposition~\ref{Bdd}). Similar results hold for their Harder-Narasimhan factors (Lemma~\ref{HNGrBdd}).
\item The three assertions in 
Lemma~\ref{EquiSS2}  
  are equivalent.
\end{enumerate}
Fix such an $m$ and let $V$ be a vector space such that
\[\dim V=P(m).\]
 
Suppose $(E,\alpha)$ is a semistable pair, then $E$ can be viewed as a quotient 
\[q:V\otimes \oo_{X}(-m)\twoheadrightarrow E.\]
Another datum of the pair is the morphism $\alpha$. It gives rise to a linear map 
\[\sigma: H^{0}(E_{0}(m))\to H^{0}(E(m))\cong V.\]
Thus, a semistable pair gives rise to the following diagram 
 \[\begin{tikzcd}K_{0} \arrow[hookrightarrow]{r}{\iota} & H^{0}(E_{0}(m))\otimes\oo_{X}(-m)\arrow[two heads]{r}{{\rm ev}}\arrow{d}{\sigma}& E_{0}\arrow{d}{\alpha}\\
 \, & V\otimes\oo_{X}(-m)\arrow[two heads]{r}{q}& E\end{tikzcd}.\]
 Here, $\iota$ is the kernel of the evaluation map ${\rm ev}$.
Conversely, we can obtain a pair from a quotient $q$ and a linear map $\sigma$ as long as 
\[q\circ \sigma\circ \iota=0.\] 
  Also notice that, $\sigma=0$ if and only if $\alpha=0$.

We will study in the following spaces:
\begin{eqnarray*}
\p= \p(\Hom(H^{0}(E_{0}(m)),V))={\rm Proj}(H^{0}(E_{0}(m))\otimes V^{\vee})\\
\mbox{and}\quad Q=\quot_{X}^{P}(V\otimes\oo_{X}(-m)).\quad\quad\quad\quad\quad
\end{eqnarray*}
This is motivated by a similar construction in~\cite{MR1299005,MR1316305}.
Spaces $\p$ and $Q$ are fine moduli spaces, with universal families:
\begin{eqnarray}
\label{universal-family-p}H^{0}(E_{0}(m))\otimes \oo_{\p}\to V\otimes \oo_{\p}(1)\\
\label{universal-family-q}\mbox{and}\quad V\otimes\oo_{X}(-m)\to \mathscr E.
\end{eqnarray}
Let 
\[Z\subset \p\times Q\]
be the locally closed subscheme of points $\xi=([\sigma],[q])$ such that 
\begin{enumerate}[(i)]
\item $q\circ \sigma\circ \iota=0$;
\item $E$ is pure;
\item the quotient $q$ induces an isomorphism of vector spaces 
\[V\stackrel{\sim}{\to} H^{0}(E(m)).\]
\end{enumerate}

There is a natural $\SL(V)$-action on $\p\times Q$, given as follows,
\[([\sigma],[q]).g=([g^{-1}\circ\sigma],[q\circ g])\]
for $g\in\SL(V)$ and $([\sigma],[q])\in \p\times Q$. It can be easily checked that this indeed defines a right action. It is clear that $Z$ is invariant under this action. The closure $\bar Z$ of $Z\subset \p\times Q$ is invariant as well.

We are going to construct the moduli space by taking the GIT quotient of $\bar Z$, eliminating the extra information coming from identifying $V$ and $H^{0}(E(m))$. 
A key step is to relate the $\delta$-stability condition to a GIT-stability condition, which will occupy a large part of this section.
The central theorem we will need is the Hilbert-Mumford Criterion, which we shall recall.

Let $G$ be an algebraic group. A {\it 1-parameter subgroup} (1-PS) of $G$ is a nontrivial homomorphism
\[\lambda: \mathbb G_{m}\to G.\]
Let $X$ be a proper algebraic scheme with a $G$-action, and $L$ be a $G$-linearized line bundle. Let $x\in X$ be a closed points, and denote the orbit map as
\[\psi_{x}:G\times \{x\}\to X.\]
Let $\lambda$ be a 1-PS of $G$. Then $\psi_{x}\circ \lambda$ 
extends to a morphism
\[f:\A^{1}\to X.\]
The point $f(0)$ is fixed under the action of $\mathbb G_{m}$. Thus, there is an induced action of $\mathbb G_{m}$ on the fiber of $L$ at $f(0)$, which is nothing but a character of $\mathbb G_{m}$:
\[\mathbb G_{m}\to \mathbb G_{m}, \quad z\mapsto z^{r}. \]
Then we define 
\[\mu^{L}(x,\lambda)=-r.\]
Although we use the terminologies invertible sheaves and line bundles interchangeably, we do mean the line bundle associated to the invertible sheaf in defining $\mu^{L}(x,\lambda)$.

GIT semistability can be determined by numbers $\mu^{L}(x,\lambda)$ via the following practical criterion:
\begin{customthm}{HM}[Hilbert-Mumford Criterion]\label{hilbert-mumford} Let a reductive group $G$ act on a proper scheme $X$. Let $L$ be an ample $G$-linearized line bundle, and $x\in X$ a closed point. Then
\begin{eqnarray*}
x\mbox{ is semistable w.r.t. }L&\iff& \mu^{L}(x,\lambda)\geq 0,\,\,\forall\mbox{1-PS }\lambda, \\
x\mbox{ is stable w.r.t. }L&\iff& \mu^{L}(x,\lambda)> 0,\,\,\forall\mbox{1-PS }\lambda.
\end{eqnarray*}

\end{customthm}

We next define a $\SL(V)$-linearized ample line bundle on $\p\times Q$, with respect to which, the GIT-stability condition will agree with the $\delta$-stability condition.

For an $l\gg 0$, we obtain an embedding, which is $\SL(V)$-equivariant,
\begin{eqnarray*}
Q=\quot_{X}^{P}(V\otimes\oo_{X}(-m))&\hookrightarrow& \grass(V\otimes H^{0}(\oo_{X}(l-m)),P(l)),\\
\lbrack q:V\otimes\oo_{X}(-m)\twoheadrightarrow E\rbrack&\mapsto& [H^{0}(q(l)):V\otimes H^{0}(\oo_{X}(l-m))\twoheadrightarrow H^{0}(E(l))].\end{eqnarray*}
A priori, we may just let $l=m$. But in order to make the calculation easier, we want $l$ to be large. The reason will be clear later.

The standard very ample line bundle on the Grassmannian is $\SL(V)$-linearized. 
Let $\oo_{Q}(1)$ be its pullback to $Q$. The line bundle $\oo_{\p}(1)$ is also $\SL(V)$-linearized. For positive integers $n_{1}$ and $n_{2}$, consider the $\SL(V)$-linearized line bundle 
\[L=\oo_{\p}(n_{1})\boxtimes\oo_{Q}(n_{2}).\]
Let $\lambda:\mathbb C^{*}\to \SL(V)$ be a 1-PS of $\SL(V)$. For a point 
\[\xi=([\sigma],[q])\in \p\times Q,\]
 we will compute $\mu^{L}(\xi,\lambda)$. Recall that
 \[\mu^{L_{1}\otimes L_{2}}=\mu^{L_{1}}+\mu^{L_{2}}.\]
 So, we can calculate $\mu^{\oo_{\p}(n_{1})}(\xi,\lambda)$ and $\mu^{\oo_{Q}(n_{2})}(\xi,\lambda)$ separately.
 
 From the $\mathbb C^{*}$-action on $V$, we have a weight decomposition of $V$ as 
 \[V=\bigoplus_{1\leq i\leq s}V_{i}\]
  where $V_{i}$ is the isotypic component of weight $\gamma_{i}\in \Z$. Arrange $\gamma_{i}$'s such that $\gamma_{1}< \gamma_{2}<\cdots<\gamma_{s}$. We get a filtration
\[W_{1}\subsetneqq W_{2}\subsetneqq\cdots\subsetneqq W_{s}=V,\]
where $W_{j}=\oplus_{i=1}^{j}V_{i}$.

Let $\{e^{i}_{u}\}_{u}$
 be a basis of $V_{i}$. Then $\sigma:H^{0}(E_{0}(m))\to V$, considered as an element in $H^{0}(E_{0}(m))^{\vee}\otimes V$, can be written as 
 \[\sigma=\bigoplus_{i,u}f_{u}^{i}\otimes e_{u}^{i}\] 
 where 
 $0\not=f_{u}^{i}\in H^{0}(E_{0}(m))^{\vee}.$
  Denote the largest $i$ appearing in the direct sum by $i(\sigma)$ and define 
 $\gamma(\sigma)=\gamma_{i(\sigma)}$. Then, the contribution to $\mu^{L}(x,\lambda)$ from $\oo_{\p}(n_{1})$ is 
 \[n_{1}\gamma(\sigma).\]

The filtration on $V$ induces a filtration on $E$
\[F_{1}\subsetneqq F_{2}\subsetneqq\cdots\subsetneqq F_{s}=E.\]
Here $F_{j}= q(W_{j}\otimes \oo_{X}(-m))$.
 Let $\gr_{j}=F_{j}/F_{j-1}$.
 The family of subsheaves $F\subset E$ of the form $q(W\otimes\oo_{X}(-m))$ for some subspace $W\leq V$, is bounded. So is the family of factors. Thus, $\exists l_{0}\in \N$, such that for all such $F_{j}$'s and factors,
 \begin{equation}\label{choice-l-plucker-embedding}H^{i}(F_{j}(l))=0 \mbox{ and }H^{i}(\gr_{j}(l))=0,\quad \forall l>l_{0},\, i>0.\end{equation}
 Twist the filtration of $E$ by $\oo_{X}(l)$, we get a filtration of $E(l)$, whose factors are $\gr_{i}(l)$. 
 
 View $q$ as a point in the Grassmannian, the limit
 \[\lim_{z\to 0}q.\lambda(z)= \bigoplus_{i=1}^{s}H^{0}(\gr_{i}(l)).\]
For a proof of this claim, see \cite[Lemma 4.4.3]{huybrechts-lehn-moduli-of-sheaves}. The line bundle associated to the invertible sheaf $\oo_{Q}(1)$ at this limit has fiber 
 \[\bigwedge ^{P(l)}\bigoplus_{i=1}^{s}H^{0}(\gr_{i}(l)).\]
 This has weight
 \[\sum_{i=1}^{s}\gamma_{i}h^{0}(\gr_{i}(l)).\]
 Thus, the contribution to $\mu^{L}(\xi,\lambda)$ from $\oo_{Q}(n_{2})$ is
\begin{equation*}
-n_{2}\sum_{i}\gamma_{i}h^{0}(\gr_{i}(l))=-n_{2}\sum_{i=1}^{s}\gamma_{i}(P_{F_{i}}(l)-P_{F_{i-1}}(l))
\end{equation*}
 We would like to point out that, when calculating the contribution from $\oo_{\p}(n_{1})$, we look at sub-line bundles. But this time, we study quotient bundles, which is more convenient here. 

Thus, we have
\begin{lem}\label{mu}With notations as above, then
\[\mu^{L}(\xi,\lambda)=n_{1}\gamma(\sigma)-n_{2}\sum_{i=1}^{s}\gamma_{i}(P_{F_{i}}(l)-P_{F_{i-1}}(l)).\]\end{lem}

\begin{lem}\label{ReIntGIT}For $l$ large as in (\ref{choice-l-plucker-embedding}), let 
\[\xi=([\sigma],[q])\in \bar Z,\] 
be a point with associated morphism $\alpha:E_{0}\to E$.
Then the following two conditions are equivalent:
\begin{enumerate}[(i)]
\item $\xi$ is GIT-stable with respect to $ L$;
\item For any nontrivial proper subspace $W< V$, let  
\[G=q(W\otimes \oo_{X}(-m)).\]
Then
\begin{equation}\label{WG}P_{G}(l)>\frac{n_{1}}{n_{2}}\Big(\epsilon_{W}(\sigma)-\frac{\dim W}{\dim V}\Big)+P(l)\frac{\dim W}{\dim V}.\end{equation}
Here, $\epsilon_{W}(\sigma)$ is either $1$ of $0$ depending on whether $W$ contains $\im \sigma$ or not.
\end{enumerate}\end{lem}
\begin{proof}Suppose $\xi$ is GIT-stable. Let $W$ and $G$ be as in the statement. 
Let 
\[k=\dim W.\] Consider the one parameter subgroup given as follows
\begin{eqnarray*}
\lambda(t)=\left( \begin{array}{cc}
  t^{k-P(m)}\Id_{k}&  \\
 &  t^{k}\Id_{P(m)-k}
\end{array} \right),
\end{eqnarray*}where $t$ acts on $W$ be by multiplying $t^{k-P(m)}$ and its complementary space by multiplying $t^{k}$.
If $\im \sigma\subset W$, then by Theorem~\ref{hilbert-mumford} and Lemma~\ref{mu}, \begin{eqnarray*}
0<\mu^{L}(\xi,\lambda)=n_{1}\big(k-P(m)\big)-n_{2}kP(l)+n_{2}P(m)P_{G}(l).
\end{eqnarray*}
Recall that $\dim V=P(m)$. Thus, in this case,
\[P_{G}(l)>\frac{n_{1}}{n_{2}}\Big(1-\frac{\dim W}{\dim V}\Big)+P(l)\frac{\dim W}{\dim V}.\]
If $\im \sigma \nsubset W$, 
then 
\[0<\mu^{L} (\xi,\lambda)=n_{1}k-n_{2}kP(l)+n_{2}P(m)P_{G}(l).\] In this case,
\[P_{G}(l)>-\frac{n_{1}}{n_{2}}\cdot\frac{\dim W}{\dim V}+P(l)\frac{\dim W}{\dim V}.\]
We thus have proven one direction of the statement.

Suppose we have the inequality~(\ref{WG}). We use the notations in the discussion right before Lemma~\ref{mu}. 
The inequality, combined with Lemma~\ref{mu}, implies
\begin{eqnarray*}
\mu^{L}(\xi,\lambda)
> n_{1}\gamma_{s}-n_{2}\gamma_{s}P(l)+\Big(-\frac{n_{1}}{P(m)}+\frac{n_{2}P(l)}{P(m)}\Big)\sum_{i=1}^{s-1}(\gamma_{i+1}-\gamma_{i})\dim W_{i}
\end{eqnarray*}
Moreover,
\begin{eqnarray*}
\sum_{i=1}^{s-1}(\gamma_{i+1}-\gamma_{i})\dim W_{i}=\gamma_{s}P(m).
\end{eqnarray*}
Thus, $\mu^{L}(\xi,\lambda)> 0$. Therefore, $\xi$ is GIT-stable.
\end{proof}
GIT-semistability can also be characterized by the corresponding weak inequality.

Now, let 
\begin{equation}\label{linearized-line-bundle}\frac{n_{1}}{n_{2}}=\frac{P(l)}{2r}.\end{equation}

We fix an $l$ such that 
\begin{enumerate}[(i)]
\item (\ref{choice-l-plucker-embedding}) holds;
\item (\ref{WG}) holds if and only if it holds as an inequality of polynomials in $l$:
\begin{equation}\label{WG-polynomial}
P_{G}>\frac{n_{1}}{n_{2}}\Big(\epsilon_{W}(\sigma)-\frac{\dim W}{\dim V}\Big)+P\frac{\dim W}{\dim V}.\end{equation}
\end{enumerate}
We can ask for the last condition because the family of such $G$'s is bounded.

\begin{cor}\label{InjGlobal} 
If $([\sigma],[q])\in \bar Z$ is GIT-semistable, then 
\[H^0(q(m)):V\to H^{0}(E(m))\]
is injective and for any coherent subsheaf $G\subset E$ such that $\dim G\leq d-1$, $ H^{0}(G(m))=0$.
\end{cor}
In defining $Z$, we require the quotient to be pure. When we take the closure, we may include non-pure quotients. But this lemma imposes restrictions.
\begin{proof}Let $W$ be the kernel of $H^0(q(m)):V\to H^0(E(m))$, then the image 
\[G=q( W\otimes \oo_X(-m))=0.\]
 The inequality (\ref{WG-polynomial}) forces $\dim W$ to be zero, otherwise the right-hand side of the inequality is a positive polynomial while the left-hand side is $0$.

Suppose $G\subset E$ such that $\dim G\leq d-1$. Let $W=H^{0}(G(m))$, then $q(W\otimes\oo_X(-m))\subset G$.
 By the inequality (\ref{WG-polynomial}), we have $\dim W=0$, otherwise the right-hand side will be a positive polynomial of degree no less than $d$, while the left hand side is of degree $\leq d-1$.\end{proof}

We are ready to relate the $\delta$-stability condition to the GIT-stability condition.
\begin{prop}\label{gitss-semistable}Let $([\sigma],[q])$ and $(E,\alpha)$ be as in the previous lemma. 
The following two assertions are equivalent
\begin{enumerate}[(i)]
\item $([\sigma],[q])$ is GIT-(semi)stable with respect to $L$;
\item $(E,\alpha)$ is (semi)stable and $q$ induces an isomorphism $V\stackrel{\sim}{\to} H^0(E(m))$.
\end{enumerate}\end{prop}
Recall that when $\deg \delta\geq \deg P$, there are no strictly semistable pairs.
\begin{proof}
First, assume that a point $([\sigma],[q])\in \bar Z$ is GIT-semistable. Denote the quotient by $q:V\otimes\oo(-m)\to E$. Then by Corollary~\ref{InjGlobal},  we know that the induced linear map $V\to H^0(E(m))$ is injective. Since $([\sigma],[q])$ is in the closure of $Z$, $E$ can be deformed to a pure sheaf. By \cite[Proposition 4.4.2]{huybrechts-lehn-moduli-of-sheaves}, there is an exact sequence
\[0\to T_{d-1}(E)\to E\stackrel{\phi}{\to} F\]where $T_{d-1}(E)$ is the torsion of $E$ and such that $P_F=P_E=P$. According to Corollary~\ref{InjGlobal}, the exact sequence provides an injective linear map 
\[H^0(E(m))\hookrightarrow H^0(F(m)).\]

For any dimension $d$ quotient $\pi:F\twoheadrightarrow F''$, let $G$ be the kernel of $\pi\circ \phi$
\[0\to G\to  E\stackrel{\pi\circ \phi}{\longrightarrow} F''\to 0.\] 
Let 
\[W=V\cap H^0(G(m)).\]
 Then we have
 \begin{equation}\label{GEF''}
h^0(F''(m))\geq h^0(E(m))-h^0(G(m))\geq \dim V-\dim W.
\end{equation}

Let $r''=r(F'')$. Let's consider the leading coefficients of two sides of (\ref{WG}), viewed as polynomials in $l$. (This is where the argument diverges, depending on the degree of $\delta$. Here, we focus on the case where $\deg\delta\geq d$.) 
Then
\begin{eqnarray}\label{VW}(2r(G)-\epsilon_W(\sigma))\dim V\geq(2r-1)\dim W.\end{eqnarray}
Combining (\ref{GEF''},\ref{VW}), we have
\begin{eqnarray*}\frac{h^0(F''(m))}{2r''-\epsilon(\pi\circ\phi\circ\alpha)}
\geq\frac{\dim V}{2r-1}\cdot\frac{2r''-(1-\epsilon_W(\sigma))}{2r''-\epsilon(\pi\circ\phi\circ\alpha)}
\geq\frac{P(m)}{2r-1}\end{eqnarray*}
To prove the second inequality, notice that, when 
\[\epsilon(\pi\circ\phi\circ\alpha)=0,\]
 $\im \alpha\subset G$. Therefore $\im\sigma\subset H^0(G(m))$. Thus, $\im\sigma\subset W$. 
 
 According to Lemma~\ref{EquiSS2}, the pair $(F,\phi\circ\alpha)$ is semistable. Therefore, by our choice of $m$, $h^0(F(m))=P(m)$. We have the following commutative diagram 
 \[\begin{tikzcd}
 V\otimes\oo_{X}(-m)\arrow{r}{\sim} \arrow{rd}{q} 
 &H^0(E(m))\otimes\oo_{X}(-m)\arrow{r}{\sim} \arrow{d}{\rm ev} 
 &H^0(F(m))\otimes\oo_{X}(-m)\arrow[two heads]{d}{\rm ev}\\
 &E\arrow{r}{\phi} &F
 \end{tikzcd}\]
  
So $\phi$ is surjective. Since they have the same Hilbert polynomial, it is an isomorphism. Therefore, $(E,\alpha)$ is a semistable pair.

Next, we assume that $(E,\alpha)$ is semistable, thus stable, and $q(m)$ induces an isomorphism between global sections. For any nontrivial proper subspace $W< V$, let
\[ G= q(W\otimes\oo(-m))\]
 and $(G,\alpha^{\prime})$ the corresponding sub-pair. If $(G,\alpha^{\prime})=(E,\alpha)$, the inequality in Lemma~\ref{ReIntGIT} holds. Assume that $(G,\alpha^{\prime})$ is a proper sub-pair. According to Lemma~\ref{EquiSS2}, we have
\[\frac{h^{0}(G(m))}{2r(G)-\epsilon(\alpha^{\prime})}< \frac{h^{0}(E(m))}{2r-1}.\]
From the commutative diagram
\[\begin{tikzcd}
W\arrow{r}\arrow[hook]{d}&H^0(G(m))\arrow[hook]{d}\\
V\arrow{r}{\cong}&H^0(E(m))\end{tikzcd},\]we know that $\dim W\leq h^{0}(G(m))$. Thus,
\[\frac{\dim W}{2r(G)-\epsilon(\alpha^{\prime})}< \frac{h^{0}(E(m))}{2r-1}.\]Therefore,\[r(G)>\frac{1}{2}\epsilon(\alpha^{\prime})-\frac{1}{2}\cdot\frac{\dim W}{\dim V}+r\frac{\dim W}{\dim V},\]which implies the inequality in Lemma~\ref{ReIntGIT}, since $\epsilon(\alpha^{\prime})\geq \epsilon_{W}(\sigma)$. Hence, $([\sigma],[q])$ is GIT-stable.
\end{proof}
We still need the following lemma, which will help us identify closed orbits. A pair is {\it polystable} if it is isomorphic to a direct sum of stable pairs, degenerate or not, with the same reduced Hilbert polynomial. 
\begin{lem}
\label{closed-orbits}The closures in $\bar Z^{ss}$ of orbits of two points, $([\sigma_{1}],[R_{1}])$ and $([\sigma_{2}],[R_{2}])$, intersect if and only if their associated semistable pairs $(E_{1},\alpha_{1})$ and $(E_{2},\alpha_{2})$ have the same Jordan-H\"older factors. The orbit of a point $([\sigma],[q])$ is closed if and only if the associated pair $(E,\alpha)$ is polystable. 
\end{lem}
The proof is similar to that of \cite[Theorem 4.3.3]{huybrechts-lehn-moduli-of-sheaves}, using the following lemma on semicontinuity.

\begin{lem}[Semicontinuity]\label{semicontinuity}
Suppose $(\mathscr F, \alpha)$ and $(\mathscr G, \beta)$ over $X_{T}=T\times  X$ are two flat families of pairs, with Hilbert polynomials $P_{\mathscr F}$ and $P_{\mathscr G}$, parametrized by a scheme $T$ of finite type over $k$. Then, the following function is semicontinuous:
\[t\mapsto \dim_{k}\Hom_{ \{t\}\times X}((\mathscr F_{t},\alpha_{t}),(\mathscr G_{t},\beta_{t})).\]
\end{lem}

The proof is modified from that of \cite[Lemma 3.4]{MR1316305}.
\begin{proof}The space $\Hom((\mathscr F_{t},\alpha_{t}),(\mathscr G_{t},\beta_{t}))$ is related to the pullback in the following diagram
\begin{equation*}
\begin{tikzcd}
C_{t}
\arrow[dashed]{r} \arrow[dashed]{d}
& k\arrow{d}{\cdot \beta_{t}}\\
\Hom(\mathscr F_{t},\mathscr G_{t}) \arrow{r}{\circ \alpha_{t}}
& \Hom(E_{0}, \mathscr G_{t})
\end{tikzcd},
\end{equation*}
in the sense that it satisfies the following equality
\[\dim \Hom((\mathscr F_{t},\alpha_{t}),(\mathscr G_{t},\beta_{t}))=\dim C_{t}-1+\epsilon(\beta_{t}).\]
By our flatness assumption, $\beta_{t}$ is either always zero or never zero.  
Thus, it is enough to show that $C_{t}$ is a fiber of a common coherent $\oo_{T}$-module, as $t$ varies.

Since the question is local on $T$, assume $T=\spec A$, where $A$ is a $k$-algebra. 

It is shown in the proof of \cite[Lemma 3.4]{MR1316305} that, there is a bounded above complex $M_{E_{0}}^{\bullet}$ of finite type free $A$-modules, such that for any $A$-module $M$,
\begin{equation}\label{complex-ext-eknot}
h^{i}(M_{E_{0}}^{\bullet}\otimes_{A}M)\cong \Ext^{i}_{X_{T}}(\pi_{2}^{*}E_{0}, \mathscr G\otimes_{A} M).
\end{equation}
Similarly, there is such an $M_{\mathscr F}^{\bullet}$ that 
\begin{equation}\label{complex-ext-f}
h^{i}(M_{\mathscr F}^{\bullet}\otimes_{A}M)\cong \Ext^{i}_{X_{T}}(\mathscr F, \mathscr G\otimes_{A} M).
\end{equation}
The morphism $\alpha$ induces a morphism of complexes, which is still denoted as $\alpha:M_{\mathscr F}^{\bullet}\to M_{E_{0}}^{\bullet}$. The morphism $\beta$ induces a morphism $\beta:A\to M_{E_{0}}^{\bullet}$. 
Thus, there is a morphism
\begin{equation*}\psi=(\alpha,-\beta):M_{\mathscr F}^{\bullet}\oplus A\to M_{E_{0}}^{\bullet}.
\end{equation*}
Then the mapping cone $C(\psi)$ fits in the following distinguished triangle
\begin{equation*}
C(\psi)[-1]\to M_{\mathscr F}^{\bullet}\oplus A\to M_{E_{0}}^{\bullet} \to C(\psi) .
\end{equation*}
Taking the long exact sequence, we have
\begin{equation*}
0\to h^{-1}(C(\psi))\to \Hom_{X_{T}}(\mathscr F, \mathscr G)\oplus A\to \Hom_{X_{T}}(\pi_{2}^{*}E_{0}, \mathscr G)\to 
\end{equation*}
Thus, we have the following fiber diagram
\begin{equation*}
\begin{tikzcd}
h^{-1}(C(\psi))
\arrow{r} \arrow{d}
& A \arrow{d}{ \beta}\\
\Hom_{X_{T}}(\mathscr F, \mathscr G) \arrow{r}{ \alpha}
& \Hom_{X_{T}}(\pi_{2}^{*}E_{0}, \mathscr G)
\end{tikzcd}.
\end{equation*}
Therefore, together with (\ref{complex-ext-eknot},\ref{complex-ext-f}) and the isomorphism
\[\Ext^{i}_{X_{T}}(\mathscr F,\mathscr G\otimes k(t))\cong \Ext^{i}_{X_{t}}(\mathscr F_{t},\mathscr G_{t}),\]
we know that
\[C_{t}\cong h^{-1}(C(\psi))\otimes k(t).\]
This finishes the proof.
\end{proof}We can now prove the existence of the moduli space.
\begin{proof}[Proof of Theorem~\ref{MainThm}]
Let 
\[S=S_X(P,\delta)= \bar{Z}\sslash SL(V)\]
 be the GIT quotient. This is a projective scheme. We will show that this is the coarse moduli space of S-equivalence classes of semistable pairs.

Suppose we are given a family of semistable pairs parametrized by $T$ 
\[\beta: \pi_2^*E_0\to \mathscr F,\]
 where $\pi_i$ is the projection from $T\times X$ onto the $i$-th factor.
Let $m$ be chosen as before, then $\pi_{1*}(\mathscr F(m))$ is locally free of rank $P(m)=\dim V$ and we acquire a morphism over $T$
\[\pi_{1*}(\beta(m)):\pi_{1*}(\pi_{2}^*E_0(m))\to \pi_{1*}(\mathscr F(m)).\]
Therefore, there is an open affine cover $T=\cup T_i$, such that over each $T_i$, $\pi_1(\mathscr F(m))|_{T_i}$ is free of rank $P(m)$. 
Choose an isomorphism over $T_i$
\[\omega_i:V\otimes\oo_{T_i}\to \pi_{1*}(\mathscr F(m))|_{T_i}.\]
 Then $\omega_i^{-1}\circ\pi_{1*}(\beta(m))$ induces a morphism $T_i\to \p$.
Also, the quotient 
\[\mathrm{ev}\circ \pi^*_1(\omega_i):V\otimes\oo_X(-m)\stackrel{\cong}{\to}\pi_1^*\pi_{1*}(\mathscr F(m))\otimes \oo_X(-m)\twoheadrightarrow \mathscr F\]
 over $T_i\times X$
 induces a morphism $T_{i}\to Q$.
Thus, they induce a morphism 
\[f_i:T_i\to \p\times Q.\]
 By the definition of $Z$ and Proposition~\ref{gitss-semistable}, $f_{i}$ factors through $\bar Z^{ss}$.
Therefore, we obtain unambiguously a morphism 
\[f_{\beta}:T\to S.\]
Thus, we have a natural transformation 
\[\mathcal{S}=\mathcal{S}_X(P,\delta)\to \mathrm{Mor\,}(-,S).\]

Suppose there is a natural transformation 
\begin{equation}\label{coarse-natural}\mathcal{S}\to \mathrm{Mor\,}(-,N).\end{equation}
Let $T=\bar{Z}^{ss}$. Universal families (\ref{universal-family-p},\ref{universal-family-q}) induce 
\[H^0(E_0(m))\otimes \oo_{X}(-m)\to V\otimes \oo_{\p}(1) \otimes \oo_{X}(-m)\twoheadrightarrow \mathscr E\otimes \oo_{\p}(1).\]Over $T$, the composition induces a family
\begin{equation}\label{family-before-quotient}
\pi^{*}_{2}E_{0}\to \mathscr E\otimes \oo_{\p}(1),
\end{equation}
thus an element in $\mathcal S(T)$. This in turn produces a map 
\[T=\bar{Z}^{ss}\to N.\] 
Because the transformation (\ref{coarse-natural}) is natural, this map is $\SL(V)$-equivariant, with the action on $N$ being trivial.
According to properties of a quotient, the map factors uniquely through $S$.
Therefore, we have the following commutative diagram of functors
\[\begin{tikzcd}
\mathcal{S} \arrow{r} \arrow{dr} &\mathrm{Mor\,}(-,S)\arrow{d}\\
& \mathrm{Mor\,}(-,N).
\end{tikzcd}
\]
Moreover, closed points in $S$ are in bijection with S-equivalence classes of semistable pairs, according to Lemma~\ref{closed-orbits}.
Thus, $S$ is the coarse moduli space.

Let us consider the open set $\bar{Z}^s\subset\bar{Z}^{ss}$ of stable points. 
The geometric quotient 
\[\bar{Z}^s\to \bar{Z}^s/{\rm SL}(V)=S^{s}(P,\delta)=S^{s}\] provides a quasi-projective scheme parametrizing equivalence classes of stable pairs. 
We shall prove this quotient to be a principal $\mathrm{PGL}(V)$-bundle. 
It is enough to show that the stabilizers are products of the identity matrix and roots of unity.

 Suppose a point $([\sigma],[q])\in \bar Z^s$ gives rise to a stable pair $\alpha:E_0\to E$ and $([\sigma],[q])$ is fixed by $g\in \SL(V)$, that is,
\[[\sigma]=[g^{-1}\circ\sigma]\mbox{\quad and \quad}[q]=[q\circ g].\]
Then there is a scalar $a\in k^{\times}$, such that $g^{-1}\circ \sigma=a\sigma$, and there is an isomorphism $\phi:E\to E$, such that $\phi\circ q=q\circ g$. Therefore, 
\[\phi\circ\alpha\circ{\rm ev}=a\alpha\circ{\rm ev}:H^{0}(E_{0}(m))\otimes \oo_{X}(-m)\to E.\]
So, $\phi\circ\alpha=a\alpha$. 
Thus, $\phi$ is a multiplication by a nonzero scalar, by Lemma~\ref{SmallLarge}. 
In the following diagram
\[\begin{CD}V@>H^0(q(m))>>H^0(E(m))\\
@VgVV @VVH^0(\phi(m))V\\
V@>H^0(q(m))>>H^0(E(m)),\end{CD}\]
 the horizontal arrows are isomorphisms and the right vertical arrow is a multiplication by a nonzero scalar. Therefore, $g$ is also a multiplication by a nonzero scalar. Because $g$ lies in $\SL(V)$, it is the product of a root of unity and the identity matrix.

In the family (\ref{family-before-quotient}), $\mathscr E\otimes \oo_{\p}(1)$ is $\SL(V)$-equivariant.
 Although the actions of the center of $\SL(V)$ on $ \oo_{\p}(1)$ and $\mathscr E$ are not trivial,
  its action on $\mathscr E\otimes \oo_{\p}(1)$ is. Thus, $\mathscr E\otimes \oo_{\p}(1)$ is $\rm{PGL}(V)$-equivariant.
Therefore, the restriction of (\ref{family-before-quotient}) to $\bar Z^{ss}\times X$ descends to $S^s\times X$ to give a universal family of pairs. 
Hence, $S^s$ represents the functor $\mathcal{S}^s_X(P,\delta)$.
\end{proof}

\vskip45pt

\section{Deformation and obstruction theories}\label{deformation-theory}
\vskip35pt

This section is devoted to the proof of Theorem~\ref{deformation-obstruction}, following \cite{huybrechts-lehn-moduli-of-sheaves,MR1966840}. In \ref{constructions}, we will outline the construction of the obstruction class and identify the deformation space. In \ref{proofs}, we will fill in the proofs.
\vskip20pt
\subsection{Constructions}\label{constructions}

Suppose $(E,\alpha)$ is a stable pair and
 \begin{equation*}
 0\to K\to B\stackrel{\sigma}{\to} A\to 0
 \end{equation*}
is a short exact sequence, where $A,B\in \mathcal Art_{k}$, such that $\frak m_{B}K=0$. 
Suppose
\[\alpha_{A}:E_{0}\otimes A\to E_{A}\]
over $X_{A}=X\times \spec A$ 
is a (flat) extension of $(E,\alpha)$. Let 
\[I^{\bullet}_{A}=\{E_{0}\otimes A\to E_{A}\}\]
 denote the complex positioned at $0$ and $1$. We would like to extend $(E_{A},\alpha_{A})$ to a pair $(E_{B},\alpha_{B})$ over $X_{B}$. This is similar to deforming a sheaf or a perfect complex. But we need to fix $E_{0}$.

We take two locally free resolutions $P^{\bullet}\stackrel{\sim}{\to } E_{0}$ and $Q_{A}^{\bullet}\stackrel{\sim}{\to } E_{A}$
and lift $\alpha_{A}$ to a morphism of complexes $\alpha^{\bullet}_{A}:P^{\bullet}\otimes A\to Q_{A}^{\bullet}$. Then, we have the following commutative diagram

\begin{equation*}
\begin{tikzcd}
\cdots \arrow{r}{d_{P}^{-2}\otimes A} 
&  P^{-1}\otimes A \arrow{r}{d_{P}^{-1}\otimes A} \arrow{d}{\alpha_{A}^{-1}} 
&  P^{0}\otimes A \arrow{r} \arrow{d}{\alpha_{A}^{0}} 
&E_{0}\otimes A \arrow{r} \arrow{d}{\alpha_{A}}
&0\\
\cdots \arrow{r}{d_{Q_{A}}^{-2}} 
&  Q_{A}^{-1} \arrow{r}{d_{Q_{A}}^{-1}}  
&  Q_{A}^{0} \arrow{r} 
&E_{A}\arrow{r} 
&0
\end{tikzcd},
\end{equation*}
where 
\begin{equation}\label{entries-resolution}P^{i}= V^{i}\otimes \oo_{X}(-m_{i})\quad\mbox{and}\quad Q_{A}^{i}=W^{i}\otimes \oo_{X_{A}}(-n_{i}).\end{equation}
 Here, $V^{i}$ and $W^{i}$ are vector spaces and $m_{i},n_{i}\in \N$. Then, 
\begin{equation*}
Q^{\bullet}=Q_{A}^{\bullet}\otimes _{A}k\end{equation*} 
is a resolution of $E$, because $E_{A}$ is flat over $A$. 

We can view the morphism $\alpha_{A}$ as a morphism between complexes concentrated at degree $0$, then $I^{\bullet}_{A}$ can be viewed as a mapping cone
\[I^{\bullet}_{A}\cong C(\alpha_{A})[-1]\cong C(\alpha_{A}^{\bullet})[-1].\]
For the sake of notations, we write down the mapping cone explicitly:
\begin{equation*}
\cdots\to P^{-1}\otimes A\oplus Q_{A}^{-2} \stackrel{d_{A}^{-2}}{\longrightarrow}
P^{0}\otimes A\oplus Q_{A}^{-1} \stackrel{d_{A}^{-1}}{\longrightarrow}
 Q_{A}^{0}\to 0,
\end{equation*}
where 
\begin{equation}
\label{differential-cone}d_{A}^{i}=
\bigg(\begin{array}{cc}
-d_{P}^{i+1}\otimes A & 0\\
\alpha_{A}^{i+1} & d_{Q_{A}}^{i}
\end{array}\bigg).\end{equation}

We lift $d^{i}_{Q_{A}}$ to $d^{i}_{Q_{B}}$, getting a sequence
\begin{equation*}
(Q_{B}^{i},d_{Q_{B}}^{i})_{i\leq 0},\quad\mbox{where } Q_{B}^{i}=W^{i}\otimes \oo_{X_{B}}(-n_{i}).
\end{equation*}
We also lift $\alpha_{A}^{i}:P^{i}\otimes A\to Q_{A}^{i}$ to 
\begin{equation*}\alpha_{B}^{i}:P^{i}\otimes B\to Q_{B}^{i}.\end{equation*}
We then obtain a sequence 
\begin{equation}\label{cone-sequence}
(P^{i+1}\otimes B\oplus Q_{B}^{i},d_{B}^{i})_{i\leq 0},
\end{equation}
where $d_{B}^{i}$ is similar to $d_{A}^{i}$ in (\ref{differential-cone}). 
This is not necessarily a complex: 
\begin{equation}\label{differential-composition}
d^{i}_{B}\circ d^{i-1}_{B} =
\bigg(\begin{array}{cc}
0 &0 \\
-\alpha_{B}^{i+1} \circ (d_{P}^{i}\otimes B) + d_{Q_{B}}^{i}\circ \alpha^{i}_{B}
& d_{Q_{B}}^{i} \circ d_{Q_{B}}^{i-1}
\end{array}\bigg)
\mbox{ \ may not vanish.}\end{equation} 
But when it is a complex, $(Q_{B}^{\bullet},d_{Q_{B}}^{\bullet})$ forms a complex and $\alpha_{B}^{\bullet}:P^{\bullet}\otimes B\to Q_{B}^{\bullet}$ is a morphism of complexes. Thus,
\begin{equation*}H^{0}(\alpha_{B}^{\bullet}):E_{0}\otimes B\to H^{0}(Q_{B}^{\bullet},d_{Q_{B}}^{\bullet})
\end{equation*}
provides a flat extension of $\alpha_{A}$, according to Lemma~\ref{complex-implies-exact}.
 
The lower row of (\ref{differential-composition}) constitutes a map
\begin{equation}\label{degree-2-map}
P^{\bullet}[1]\otimes B\oplus Q_{B}^{\bullet}\to Q_{B}^{\bullet}[2].
\end{equation}
When restricted to $X_{A}$, it becomes zero. Moreover, $\frak m_{B}K=0$. The map above induces a map\footnote{The argument to deduce (\ref{obstruction-morphism-complex}) from (\ref{degree-2-map}) will be applied repeatedly. } 
\begin{equation}\label{obstruction-morphism-complex}
(\omega_{P}^{\bullet},\omega_{Q}^{\bullet})
:C(\alpha^{\bullet}) \to Q_{B}^{\bullet}[2]\otimes_{B} K\cong Q^{\bullet}[2]\otimes_{k} K. 
\end{equation}
We claim that $(\omega_{P}^{\bullet},\omega_{Q}^{\bullet})$ is a morphism of complexes, which will be proven, Lemma~\ref{obstruction-morphism-complex-proof}.
This induces a class, which will be shown to be the obstruction class,
\begin{equation*}
\ob(\alpha_{A},\sigma)=[(\omega_{P}^{\bullet},\omega_{Q}^{\bullet})] \in \Hom_{K(X)}(C(\alpha^{\bullet}),Q^{\bullet}[2]\otimes_{k} K).
\end{equation*}

To identify $\Hom_{K(X)}(C(\alpha^{\bullet}),Q^{\bullet}[2]\otimes K)$ with $\Ext^{1}_{D(X)}(I^{\bullet},E\otimes K)$ in the theorem, 
we only need to take (\ref{entries-resolution}) to be very negative such that
\[H^{i}(X,E(m_{j}))=0\quad \mbox{and}\quad H^{i}(X,E(n_{j}))=0,\quad\forall i>0\ \mbox{and }j\leq 0.\]
 Because, then
\begin{equation*}\Ext^{1}_{D(X)}(I^{\bullet},E\otimes K)\cong \Hom_{K(X)}(C(\alpha^{\bullet}),E[2]\otimes K) \cong\Hom_{K(X)}(C(\alpha^{\bullet}),Q^{\bullet}[2]\otimes K).
\end{equation*}

Suppose we have two extensions $\alpha_{B}:E_{0}\otimes B\to E_{B}$ and $ \beta_{B}:E_{0}\otimes B\to F_{B}$, which 
 arise from liftings 
 \[\{d_{E_{B}}^{i}:Q_{B}^{i}\to Q_{B}^{i+1},\alpha_{B}^{i}:P^{i}\otimes B\to Q_{B}^{i}\}\quad\mbox{and} \quad\{d_{F_{B}}^{i}:Q_{B}^{i}\to Q_{B}^{i+1},\beta_{B}^{i}:P^{i}\otimes B\to Q_{B}^{i}\}.\] 
The differences $d_{E_{B}}^{i}-d_{F_{B}}^{i}$ and $\alpha_{B}^{i}-\beta_{B}^{i}$ induce  a morphism of complexes
\begin{equation}\label{deformation-complex-morphism}
(f_{P}^{\bullet},f_{Q}^{\bullet}):C(\alpha^{\bullet})\to Q^{\bullet}[1]\otimes K.\end{equation}
This induces a class
\begin{equation*}
v=[(f_{P}^{\bullet},f_{Q}^{\bullet})]\in \Hom_{K(X)}(C(\alpha^{\bullet}),Q^{\bullet}[1]\otimes K)\cong \Ext^{1}_{D(X)}(I^{\bullet}, E\otimes K).
\end{equation*}
Conversely, given $\alpha_{B}$ and $(f_{P}^{\bullet},f_{Q}^{\bullet})$, we can produce another extension $\beta_{B}$.

Moreover, $\alpha_{B}$ and $\beta_{B}$ are equivalent if and only if $v=0$.

\vskip20pt
\subsection{Proofs}\label{proofs}
In this sub-section, we fill in the proofs of several claims we made in \ref{constructions}. We will assume the independence of choices in \ref{obs-cl} and provide proofs of independence in \ref{independence-choices}. To simplify the notation, we will sometimes omit the superscripts in maps between complexes, such as $\alpha^{\bullet}$ and $\alpha^{i}$.
\vskip15pt
\subsubsection{Obstruction classes}\label{obs-cl}
We first show that $\ob(\alpha_{A},\sigma)$ is an obstruction class.

Suppose an extension $(E_{B},\alpha_{B})$ exists. The definition of $\ob(\alpha_{A},\sigma)$ does not depend on the choice of the resolution $Q_{A}^{\bullet}$.
 We can assume $(E_{B},\alpha_{B})$ arises by lifting $d_{Q_{A}}^{i}$ and $\alpha_{A}^{i}$, 
making $Q_{B}^{\bullet}$ into a complex and $\alpha^{\bullet}_{B}$ a morphism of complexes. 
Then, $(\omega_{P}^{\bullet},\omega_{Q}^{\bullet})=0$.  
Thus, $\ob(\alpha_{A},\sigma)=0$.

Conversely, suppose $\ob(\alpha_{A},\sigma)=0$. It is enough to show that $(\omega_{P}^{\bullet},\omega_{Q}^{\bullet})=0$,  
after possible modifications of the liftings.
The vanishing of $\ob(\alpha_{A},\sigma)$ is equivalent to that $(\omega_{P}^{\bullet},\omega_{Q}^{\bullet})$ is homotopic to $0$. 
Let
$(g^{\bullet}_{P},g^{\bullet}_{Q})$ be a homotopy. 
By abuse of notation, let $\iota$ denote inclusions 
\[  
\iota: Q_{B}^{i}\otimes K\hookrightarrow Q_{B}^{i}.\]
Similarly, $\pi$ denotes the corresponding quotients, 
\[\pi :P^{i}\otimes B\twoheadrightarrow P^{i} \quad\mbox{and}\quad \pi: Q_{B}^{i}\twoheadrightarrow Q^{i}.\]
We can replace $\alpha_{B}$ and $d_{Q_{B}}$ by 
\begin{equation*}
\alpha_{B}-\iota\circ g_{P}\circ \pi \quad \mbox{and} \quad d_{Q_{B}}-\iota\circ g_{Q}\circ \pi, \quad \mbox{resp.,}
\end{equation*}
then the new $(\omega_{P}^{\bullet},\omega_{Q}^{\bullet})$ is zero.

 The following well-known lemma is central to our argument. For completeness, we give a proof here.
 
\begin{lem}\label{complex-implies-exact}Let $(Q_{A}^{\bullet},d^{\bullet}_{Q_{A}})$ be a sequence of the form $Q_{A}^{i}\cong W^{i}\otimes \oo_{X_{A}}(-n_{i})$, $i\leq 0$, such that
\[(Q^{\bullet}_{A},d_{Q}^{\bullet})\otimes_{A}k\cong (Q^{\bullet},d^{\bullet})\]
is a resolution of $E$. If $(Q^{\bullet}_{A},d_{Q_{A}}^{\bullet})$ is a complex, then
it is exact except at the $0$-th place and 
the cohomology $H^{0}(Q^{\bullet}_{A},d_{Q_{A}}^{\bullet})$ is an extension of $E$ flat over $A$.
\end{lem}
 
 \begin{proof}
 There is a short exact sequence of complexes
\[0\to Q_{A}^{\bullet}\otimes_{A} \frak m_{A}\to Q_{A}^{\bullet}\to Q^{\bullet}\to 0.\]

First, let $n$ be the least integer such that $\frak m_{A}^{n}=0$.
We shall show that for $0\leq i\leq n$, $Q_{A}^{\bullet}\otimes A/\frak m_{A}^{i}$ is exact except at the $0$-th place, by induction on $i$ decreasingly. Tensor $Q_{A}^{\bullet}$ over $A$ with the short exact sequence
\[0\to \frak m_{A}^{n-1}\to \frak m_{A}^{n-2}\to \frak m_{A}^{n-2}/\frak m_{A}^{n-1} \to 0,\]
whose last term is a direct sum of copies of $k$. 
On the other hand, $Q_{A}^{\bullet}\otimes \frak m_{A}^{n-1}\cong Q^{\bullet}\otimes_{k}\frak m_{A}^{n-1}$.
We deduce that the complexes 
$Q_{A}^{\bullet}\otimes\frak m_{A}^{n-1}$
 and $Q_{A}^{\bullet}\otimes\frak m_{A}^{n-2}/\frak m_{A}^{n-1}$ are exact except at the $0$-th places.
 So, from the associated long exact sequence,
 \[Q_{A}^{\bullet}\otimes\frak m_{A}^{n-2}\]
 is also exact except at the $0$-th place. 
 Inductively, we can prove this for $Q_{A}^{\bullet}$.
 
Next, let 
\[E_A=H^0(Q_{A}^\bullet,d_{Q_{A}}^{\bullet}).\]
We shall show that $E_{A}\otimes A/\frak m_{A}^{i}$
 is flat for $1\leq i\leq n$,  by induction on $i$. 

Of course $E_{A}\otimes_{A} A/\frak m_{A}\cong E$ is flat over $A/\frak m_{A}\cong k$. 
Tensor the short exact sequence 
\begin{equation}\label{ses-induction-flatness}
0\to \frak m_{A}/\frak m^{2}_{A}\to A/\frak m^{2}_{A}\to A/\frak m_{A}\to 0\end{equation}
by $Q_{A}^{\bullet}$ over $A$. 
Since the ideal $\frak m_{A}/\frak m^{2}_{A}$ is square zero, we have the short exact sequence of complexes
\[0\to Q^{\bullet}\otimes_{k} \frak m_{A}/\frak m^{2}_{A}  \to Q_{A}^{\bullet} \otimes_{A} A/\frak m^{2}_{A}\to Q^{\bullet}\to 0.\]
The associated long exact sequence degenerates to
\begin{equation}\label{ses-extension}0\to E\otimes \frak m_{A}/\frak m^{2}_{A} \to E_{A} \otimes A/\frak m_{A}^{2}\to E\to 0.\end{equation}
Therefore, 
\[E_{A}\otimes_{A} A/\frak m_{A}^{2}\]
is flat over $A/\frak m_{A}^{2}$, according to Lemma~\ref{flatness-square-zero-extension}. 

Replace (\ref{ses-induction-flatness}) by 
\[0\to \frak m_{A}^{2}/\frak m^{3}_{A}\to A/\frak m^{3}_{A}\to A/\frak m_{A}^{2}\to 0,\]
 we can repeat this argument. Inductively, we can prove $E_{A}$ is flat over $A$.

Similar to (\ref{ses-extension}), we also have the short exact sequence
\begin{equation*}0\to E_{A}\otimes \frak m_{A}\to E_{A}\to E\to 0.\end{equation*}
So, $E_{A}$ is a extension of $E$.
  \end{proof}
  
  For the readers' convenience, we include the following basic lemma about flatness. For a proof, see \cite[Proposition 2.2]{MR2583634}.
\begin{lem}\label{flatness-square-zero-extension}Let $B\to A$ be a surjective homomorphism of noetherian rings whose kernel $K$ is square zero. Then a $B$-module $M^{\prime}$ is flat over $B$ if and only if $M=M^{\prime}\otimes_{B} A$ is flat over $A$ and the natural map $M\otimes _{A}K\to M^{\prime}$ is injective.
\end{lem}

\begin{lem}\label{obstruction-morphism-complex-proof}
(\ref{obstruction-morphism-complex}) is a morphism of complexes.
\end{lem}
\begin{proof}
We have two equalities
\begin{equation}\label{def-degree-2-map}
-\alpha_{B}\circ d_{P}\otimes B+d_{Q_{B}}\circ \alpha_{B} =\iota\circ \omega_{P}\circ \pi\quad
\mbox{and}\quad d_{Q_{B}}\circ d_{Q_{B}}=\iota\circ \omega_{Q}\circ \pi.
\end{equation}
The map (\ref{obstruction-morphism-complex}) is indeed a morphism: one can show that
\[\iota\circ \Bigg( d_{Q}\otimes K \circ(\omega_{P},\omega_{Q})-(\omega_{P},\omega_{Q})
\bigg(\begin{array}{cc}
-d_{P} & 0\\
\alpha & d_{Q}
\end{array}\bigg)
\Bigg)\circ \pi=0.\]
Because $\iota$ is injective and $\pi$ is surjective, $(\omega_{P},\omega_{Q})$ commutes with differentials.\footnote{The trick using $\iota$ and $\pi$ will be applied repeatedly. }
\end{proof}

\vskip15pt
\subsubsection{Obstructions -- independence of choices}\label{independence-choices}
We now show that the $\ob(\alpha_{A},\sigma)$ is independent of the various choices we have made: $\alpha_{A}^{\bullet}$, $\alpha_{B}^{\bullet}$, $d_{Q_{B}}^{\bullet}$, and $Q_{A}^{\bullet}$.
\vskip4pt
To start, if we choose a different lifting $\alpha_{A}^{\bullet}$ of $\alpha_{A}$, then $(\omega_{P}^{\bullet},\omega_{Q}^{\bullet})$ will only differ by a homotopy.
\vskip4pt
We next show that the morphism $(\omega_{P}^{\bullet},\omega_{Q}^{\bullet})$ is independent of liftings $\alpha_{B}$ and $d_{Q_{B}}$, modulo homotopy.

Let $\alpha^{\prime}_{B}$ and $d^{\prime}_{Q_{B}}$ be different liftings, giving rise to $(\omega_{P}^{\prime\bullet},\omega_{Q}^{\prime\bullet})$. 
The differences $\alpha_{B}-\alpha^{\prime}_{B}$ and $d_{Q_{B}}-d^{\prime}_{Q_{B}}$ induce a map, which will be shown to be a homotopy,
\begin{equation*}
(h_{P}^{\bullet},h_{Q}^{\bullet}):P^{\bullet}[1]\oplus Q^{\bullet} \to Q^{\bullet}[1]\otimes_{k} K. 
\end{equation*}
The related equalities are
\begin{equation}\label{homotopy-different-liftings}\iota\circ h_{P} \circ \pi=\alpha_{B}-\alpha_{B}^{\prime}
\quad\mbox{and}\quad 
\iota\circ h_{Q}\circ \pi= d_{Q_{B}}-d^{\prime}_{Q_{B}}.
\end{equation}
Then, combining (\ref{def-degree-2-map}) and (\ref{homotopy-different-liftings}), we obtain
\begin{eqnarray*}
\omega_{P}-\omega_{P}^{\prime}&=& -h_{P}\circ d_{P} + d_{Q}\otimes K\circ h_{P} + h_{Q}\circ \alpha,\\
\omega_{Q}-\omega_{Q}^{\prime} &=& d_{Q}\otimes K \circ h_{Q} + h_{Q}\circ d_{Q}.
\end{eqnarray*}
Therefore,
\begin{equation*}
(\omega_{P},\omega_{Q})-(\omega_{P}^{\prime},\omega_{Q}^{\prime})= d_{Q}\otimes K\circ (h_{P},h_{Q}) + (h_{P},h_{Q})\bigg(
\begin{array}{cc}
-d_{P} & 0\\
\alpha & d_{Q}
\end{array}\bigg),
\end{equation*}
which means $(\omega_{P}^{\bullet},\omega_{Q}^{\bullet})$ and $(\omega_{P}^{'\bullet},\omega_{Q}^{'\bullet})$ are homotopic.
\vskip4pt
Finally, we show the independence of $Q_{A}^{\bullet}$. 

Let $(R^{\bullet}_{A},d_{R_{A}}^{\bullet})$ be another very negative resolution of the form:
\[R_{A}^{i}=W^{i\prime}\otimes \oo_{X_{A}}(-n_{i}^{\prime}).\]
 Then there is a lifting of the identity map 
$q^{\bullet}_{A}:Q_{A}^{\bullet}\to R^{\bullet}_{A}$,
 unique up to homotopy. 
 Let 
 \[\beta^{\bullet}_{A}= q_{A}^{\bullet}\circ \alpha_{A}^{\bullet}:P^{\bullet}\otimes A\to R^{\bullet}_{A}.\] 

 Moreover, there is a morphism 
 \[{\rm diag}(\id, q^{\bullet}_{A}):C(\alpha^{\bullet}_{A}) \to C(\beta^{\bullet}_{A})
 \]
Lift $q_{A}^{\bullet}$ and $\beta_{A}^{\bullet}$ to  
$q^{\bullet}_{B}:Q^{\bullet}_{B}\to R^{\bullet}_{B}$ and $\beta^{\bullet}_{B}:P^{\bullet}\otimes B\to R^{\bullet}_{B}$.
 Then, we have a map of sequences
 \[{\rm diag}(\id, q^{\bullet}_{B}):P^{\bullet}[1]\otimes B\oplus Q_{B}^{\bullet} \to P^{\bullet}[1]\otimes B\oplus R_{B}^{\bullet}.\]
 This fits in the following square, which is not necessarily commutative,
 \begin{equation}\label{square-2-resolutions}\begin{tikzcd}
 P^{\bullet}[1]\otimes B\oplus Q^{\bullet}_{B} \arrow{r}\arrow{d}[swap]{{\rm diag}(\id, q^{\bullet}_{B})}
 & Q^{\bullet}_{B}[2]\arrow{d}{q^{\bullet}_{B}}\\
 P^{\bullet}[1]\otimes B\oplus R^{\bullet}_{B} \arrow{r} 
 & R^{\bullet}_{B}[2]
 \end{tikzcd}.
 \end{equation}
Here, the two horizontal maps are as defined in (\ref{degree-2-map}).
The square above induces
 \begin{equation*}\begin{tikzcd}
 P^{\bullet}[1]\oplus Q^{\bullet} \arrow{r}{(\omega_{P}^{\bullet},\omega_{Q}^{\bullet})}
 \arrow{d}[swap]{{\rm diag}(\id, q^{\bullet})}
 & Q^{\bullet}[2]\otimes K \arrow{d}{q^{\bullet}}\\
 P^{\bullet}[1]\oplus R^{\bullet} \arrow{r}{(\bar\omega_{P}^{\bullet},\bar\omega_{R}^{\bullet})} 
 & R^{\bullet}[2]\otimes K
 \end{tikzcd}.
 \end{equation*}

To show that $\ob(\alpha_{A},\sigma)$ is independent of the resolution, it is enough to show that the two compositions differ by a homotopy. This is because, if they differ by a homotopy, two classes $[(\omega_{P}^{\bullet},\omega_{Q}^{\bullet})]$ and $[(\bar\omega_{P}^{\bullet},\bar\omega_{R}^{\bullet})]$ are identified via the isomorphism
\[\Hom_{K(X)}(C(\alpha^{\bullet}),Q^{\bullet}[2]\otimes K)\cong \Hom_{K(X)}(C(\beta^{\bullet}),R^{\bullet}[2]\otimes K).\]
Indeed, the difference $d_{R_{B}}\circ q_{B}-q_{B}\circ d_{Q_{B}}$ and $\beta_{B}-q_{B}\circ \alpha_{B}$ induce maps
\[\tau^{\bullet}:Q^{\bullet}\to R^{\bullet}[1]\otimes K \quad \mbox{and}\quad \upsilon^{\bullet}:P^{\bullet}\to Q^{\bullet}\otimes K.\]
There are the following equalities
\begin{eqnarray}
\label{tau-upsilon}d_{R_{B}}\circ q_{B}-q_{B}\circ d_{Q_{B}}=\iota \circ \tau \circ \pi
\quad \mbox{and}\quad
\beta_{B}-q_{B}\circ \alpha_{B} = \iota \circ \upsilon\circ \pi.
\end{eqnarray}

Combining (\ref{def-degree-2-map}) and (\ref{tau-upsilon}), we know that the difference of two compositions in (\ref{square-2-resolutions}) is
\begin{eqnarray*}
&&\iota\circ\big((\bar\omega_{P},\bar\omega_{R})\circ\mbox{diag}(\id, q)- q\circ (\omega_{P},\omega_{Q})\big)\circ \pi\\
&=&(-\beta_{B}\circ d_{P}\otimes B+d_{R_{B}}\circ\beta_{B},\ d_{R_{B}}\circ d_{R_{B}}\circ q_{B} )
\\
&&\quad \quad-q_{B}\circ(- \alpha_{B}\circ d_{P}\otimes B+ d_{Q_{B}}\circ\alpha_{B},  \ d_{Q_{B}}\circ d_{Q_{B}})\\
&=& \iota\circ \big(-\upsilon\circ d_{P} + \tau\circ \alpha +d_{R}\otimes K\circ \upsilon,\  \tau\circ d_{Q}+d_{R}\otimes K\circ \tau
\big)\circ \pi\\
&=&\iota\circ \Bigg ( (\upsilon, \tau)\circ\bigg(
\begin{array}{cc}
-d_{P} & 0\\
\alpha & d_{Q}
\end{array}
\bigg)+d_{R}\otimes K\circ (\upsilon, \tau)\Bigg)\circ \pi. 
\end{eqnarray*}
Thus, $(\upsilon^{\bullet},\tau^{\bullet})$ is a homotopy.

\vskip15pt
\subsubsection{Deformations} 
Assume that the obstruction class $\ob(\alpha_{A},\sigma)$ vanishes.

Suppose there are two extensions:
\[\alpha_{B}:E_{0}\otimes B\to E_{B}\quad \mbox{and} \quad \beta_{B}:E_{0}\otimes B\to F_{B}.\] 
Resolve $E_{B}$ and $F_{B}$ by two very negative complex with identical terms but different differentials:
 $(Q_{B}^{\bullet},d_{E_{B}}^{\bullet})
 $ and $ (Q_{B}^{\bullet},d_{F_{B}}^{\bullet})$.
Then, lift $\alpha_{B}$ and $\beta_{B}$ 
\[\begin{tikzcd}
P^{\bullet}\otimes B \arrow{d}{\alpha_{B}^{\bullet}} \arrow{r}{\sim} & E_{0}\otimes B\arrow{d}{\alpha_{B}}\\
(Q_{B}^{\bullet},d_{E_{B}}^{\bullet}) \arrow{r}{\sim} & E_{B}
\end{tikzcd}\quad\mbox{and}\quad
\begin{tikzcd}
P^{\bullet}\otimes B \arrow{d}{\beta_{B}^{\bullet}} \arrow{r}{\sim} & E_{0}\otimes B\arrow{d}{\beta_{B}}\\
(Q_{B}^{\bullet},d_{F_{B}}^{\bullet}) \arrow{r}{\sim} & F_{B}.
\end{tikzcd}\]
 The differences $d_{E_{B}}^{i}-d_{F_{B}}^{i}$ and $\alpha_{B}^{i}-\beta_{B}^{i}$ induce maps 
\begin{equation*}
f_{Q}^{i}:Q^{i}\to Q^{i+1}\otimes K
\quad\mbox{and}\quad
f_{P}^{i}:P^{i}\to Q^{i}\otimes K.\end{equation*}
One can show that these provide a morphism of complexes
\begin{equation}
(f_{P}^{\bullet},f_{Q}^{\bullet}):C(\alpha^{\bullet})\to Q^{\bullet}[1]\otimes K.\end{equation}
Thus, this induces a class $v$ defined by
\begin{equation*}
v=[(f_{P}^{\bullet},f_{Q}^{\bullet})]\in  
\Ext^{1}_{D(X)}(I^{\bullet}, E\otimes K).
\end{equation*}

Conversely, if we are given an extension $(E_{B},\alpha_{B})$ and a class $v$ represented by $(f_{P},f_{Q})$, then 
\[\beta_{B}=\alpha_{B}-\iota\circ f_{P}\circ \pi\quad \mbox{and}\quad d_{F_{B}}=d_{E_{B}}-\iota \circ f_{Q}\circ \pi\]
produce a morphism of complexes
$P^{\bullet}\otimes B\to (Q_{B}^{\bullet},d_{F_{B}}^{\bullet})$. 
This induces an extension of $(E_{A},\alpha_{A})$:
\[(F_{B},\beta _{B})=(H^{0}(Q_{B}^{\bullet},d_{F_{B}}^{\bullet}),H^{0}(\beta_{B}^{\bullet})).\]

If we choose a different resolution $R_{B}^{\bullet}$ 
and define $(\bar f_{P}^{\bullet},\bar f_{R}^{\bullet})$ similarly as in (\ref{deformation-complex-morphism}), 
then $[(f_{P}^{\bullet},f_{Q}^{\bullet})]$ and $[(\bar f_{P}^{\bullet},\bar f_{R}^{\bullet})]$ are identified under the isomorphism
 \[\Hom_{K(X)}(P^{\bullet}[1]\oplus Q^{\bullet},Q^{\bullet}[1]\otimes K)\cong \Hom_{K(X)}(P^{\bullet}[1]\oplus R^{\bullet},R^{\bullet}[1]\otimes K).\]
So, $v$ is independent of the resolution $Q_{B}^{\bullet}$.

 \vskip6pt
We next show that the difference of two equivalent extensions gives a zero class $v$.
 Indeed, suppose $\alpha_{B}$ and $\beta_{B}$ are equivalent, 
then by Lemma~\ref{SmallLarge}, there is a constant $z\in B$ such that $\beta_{B}=z\alpha_{B}$.
 Denote the image of $z$ in $k$ as $\bar z$.
We have proven that $v$ is independent of resolutions. So, for our convenience, we take the same resolution $Q^{\bullet}_{B}$ for $E_{B}$ and $F_{B}$, and take $\beta^{\bullet}=z\alpha^{\bullet}$. Then $f_{Q}^{\bullet}=0$. Furthermore, $f_{P}^{\bullet}$ in (\ref{deformation-complex-morphism}) is homotopic to zero via homotopy 
\[(0,1-\bar z):P^{i+1}\oplus Q^{i}\to Q^{i}\otimes K.\]
 Thus, the associated $v=0$.

\vskip6pt
It remains to prove that if $(h_{P}^{\bullet}, h_{Q}^{\bullet})$ is a homotopy between $(f_{P}^{\bullet},f_{Q}^{\bullet})$ and zero, then $\alpha_{B}$ and $\beta_{B}$ are equivalent. One can actually check that 
\begin{enumerate}[(i)]
\item $\id -\iota \circ h_{Q}\circ \pi: (Q_{B}^{\bullet}, d_{E_{B}}^{\bullet})\to (Q_{B}^{\bullet}, d_{F_{B}}^{\bullet})$ is a morphism of complexes;
\item $(\id -\iota \circ h_{Q}\circ \pi)\circ \alpha_{B}=\beta_{B}-d_{F_{B}}\circ \iota \circ h_{P}\circ \pi-\iota \circ h_{P}\circ \pi\circ d_{P}\otimes B$.
\end{enumerate}
Hence, there is a morphism $\phi$ commuting two families of stable pairs $\alpha_{B}$ and $\beta_{B}$. Therefore, by Lemma~\ref{SmallLarge}, this is an isomorphism.
\vskip45pt

\section{Stable pairs on surfaces}\label{virtual-class}
\vskip35pt

In this section, we assume that $(X,\oo_{X}(1))$ is a smooth projective surface, $E_{0}$ is torsion free, $P$ and $\delta$ are of degree $1$. We shall demonstrate that in these cases, the moduli space of stable pairs admits a virtual fundamental class, proving Theorem~\ref{virtual-fundamental-class}.

To show the existence of the virtual fundamental class, it suffices to show that the obstruction theory is {\it perfect} \cite{MR1437495,MR1467172}. 
That is, 
there is a two term complex of locally free sheaves resolving the deformation and obstruction sheaves. 
In order to do this, we essentially need to show that there are no higher obstructions, which is guaranteed by the following lemma.
\begin{lem}\label{no-higher-obstructions}Fix a stable pair $(E,\alpha)$. Then 
\[\Ext^{i}_{D(X)}(I^{\bullet},E)=0,\quad \mbox{unless } i=0,1.\]\end{lem} 
\begin{proof}The stable pair fits into an exact sequence
\begin{equation*}
0\to K\to E_{0}\to E\to Q\to 0,
\end{equation*}
which can be written as a distinguished triangle
\begin{equation*}
K\to I^{\bullet}\to Q[-1]\to K[1]. 
\end{equation*}
 Notice that $K$ is torsion free and $Q$ is $0$-dimensional.

Apply the functor $\Hom(-,E)$ to this triangle. 
The associated long exact sequence is
\newpage 
\begin{table*}[h]
\centering
\begin{tabular}{cccccccc c}
$0$ & $\to$ & $\Hom(Q,E)$ & $\to$  & $\Ext^{-1}(I^{\bullet}, E)$ & $\to$ & $0$ & $\to$ & \\
&& $\Ext^{1}(Q,E)$ & $\to$  & $\Hom(I^{\bullet}, E)$ & $\to$ & $\Hom(K,E)$ & $\to$ &\\
&&$\Ext^{2}(Q,E)$ & $\to$  & $\Ext^{1}(I^{\bullet}, E)$ & $\to$ & $\Ext^{1}(K,E)$ & $\to$ & \\
&& $0$ & $\to$ & $\Ext^{2}(I^{\bullet}, E)$ & $\to$  & $\Ext^{2}(K,E)$ & $\to $ & 0
\end{tabular}
\end{table*}

Because $Q$ is $0$-dimensional and $E$ is pure, $\Hom(Q,E)=0$. 
Therefore, $\Ext^{-1}(I^{\bullet},E)=0$.
The kernel $K$ is torsion free, so
\[\Ext^{2}(K,E)\cong \Hom(E,K\otimes \omega_{X})^{\vee}=0.\]
Thus, 
$\Ext^{2}_{D(X)}(I^{\bullet},E)=0$.
\end{proof}
Using this lemma, the expected dimension of the moduli space can be easily calculated via Hirzebruch-Riemann-Roch, knowing invariants of $E_{0}$.

Now, let 
\[\mathbb I^{\bullet}=\{\pi_{2}^{*}E_{0}\stackrel{\tilde \alpha}{\to } \mathbb E\}\]
 be the universal pair, according to Theorem~\ref{MainThm}. By Theorem~\ref{deformation-obstruction}, the deformation sheaf and the obstruction sheaf are calculated by 
\begin{equation*}\label{deformation-obstruction-complex}
R\pi_{*}R\HOM(\mathbb I^{\bullet},\mathbb E).
\end{equation*}
Take a finite complex $P^{\bullet}$ of locally free sheaves resolving $\mathbb E$ and a finite complex $Q^{\bullet}$ of very negative locally free sheaves resolving $\mathbb I^{\bullet}$. Take a finite, very negative locally free resolution $A^{\bullet}$ of $(Q^{\bullet})^{\vee}\otimes P^{\bullet}$. Then 
\begin{equation}
R\pi_{*}R\HOM(\mathbb I^{\bullet},\mathbb E)\cong
R\pi_{*}R\HOM(Q^{\bullet},P^{\bullet})
\cong R\pi_{*}A^{\bullet}.
\end{equation}
Denote this complex as $B^{\bullet}$. By Grothendieck-Verdier duality,
\begin{eqnarray*}
B^{\bullet}=R\pi_{*}A^{\bullet}&\cong& R\pi_{*}R\HOM(A^{\bullet\vee}\otimes \omega_{X},\omega_{X})\\
&\cong& R\HOM(R\pi_{*}(A^{\bullet\vee}\otimes \omega_{X})[-2],\oo)
\end{eqnarray*}
Moreover, notice that 
\[R\pi_{*}(A^{\bullet\vee}\otimes \omega_{X})=\pi_{*}(A^{\bullet\vee}\otimes \omega_{X})\]
 is a complex of locally free sheaves, due to the negativity of $A^{j}$'s. Thus, $B^{\bullet}$ is a complex of locally free sheaves as well. Denote the differentials as $d^{i}$'s.

Next, we show that $B^{\bullet}$ can be truncated to degree $0$ and $1$. 
The cohomologies of $B^{\bullet}$ concentrate at degree $0$ and $1$, by Lemma~\ref{no-higher-obstructions}. 
Suppose $B^{i\geq2}$ is the last term that is nonzero. 
Both $B^{i}$ and $B^{i-1}$ are locally free, then $\ker d^{i-1}$ is also locally free. 
Replace $B^{i}$ by zero and $B^{i-1}$ by $\ker d^{i-1}$. 
We get a new complex of locally free sheaves, which is quasi-isomorphic to $B^{\bullet}$. 
Inductively, we can trim $B^{\bullet}$ down to degree $1$. 
On the other side, suppose $B^{j< 0}$ is the first term that is nonzero. 
Then, $d^{j}$ is injective fiberwise. Therefore, $\coker d^{j}$ is flat, thus locally free. 
Hence, we can replace $B^{j-1}$ by zero and $B^{j}$ by $\coker d^{j}$ to get a new complex of locally free sheaves. Inductively, $B^{\bullet}$ becomes a complex concentrated in degree $0$ and $1$, with cohomologies the deformation sheaf and the obstruction sheaf. 
Namely, we have the following exact sequence on $S_{X}(P,\delta)$
\begin{equation*}0\to \mathcal Def\to B^{0}\to B^{1}\to \mathcal Obs\to 0,\end{equation*}
where $B^{0}$ and $B^{1}$ are locally free.

Therefore, the moduli space admits a virtual fundamental class. 

\vskip45pt


\section{Examples}\label{examples}
\vskip35pt
In this section, we study examples of moduli spaces of dimension $1$ stable pairs over K3 surfaces. 
Let $(X,\oo_{X}(1))$ be a polarized K3 surface, $P$ be a Hilbert polynomial of degree $1$, and $\delta$ be a positive polynomial of degree larger than $1$. Let $E_0$ be a fixed coherent sheaf over $X$. Then a pair $(E,\alpha) $, such that $P_{E}=P$, is stable if $E$ is pure and $\coker \alpha$ has dimension $0$, Lemma \ref{GenSurj}.

Let $H=c_{1}(\oo(1))\in H_{2}(X,\Z)$. Suppose the schematic support of $E$ has arithmetic genus $h$. There are two discrete invariants of $E$:\footnote{There is a slight abuse of notation about $\beta$ and $d$. But they are unlikely to cause confusions. }
\begin{equation}\label{chern-char}\beta_{h}=c_{1}(E)\in H_{2}(X,\Z)\quad \mbox{and} \quad \chi(E)=1-h+d.
\end{equation}
They are related to the Hilbert polynomial by $P_{E}(m)=(\beta_{h}.H)m+1-h+d$.
So, with the Hilbert polynomial fixed, there are only finitely many possible $\beta_{h}$'s. The moduli space decomposes as a disjoint union:
\begin{equation*}S_{X}^{E_{0}}(P,\delta)=\coprod_{\beta_{h}} S_{X}^{E_{0}}(\beta_{h},1-h+d)
\end{equation*}
where $S_{X}^{E_{0}}(\beta_{h},1-h+d)$ denote the moduli space of stable pairs satisfying conditions (\ref{chern-char}).

Let $C_{h}$ be a representative in the class $\beta_{h}$, then 
the linear system 
 $|C_{h}|\cong \p^{h}$.
Let 
\[\mathcal C_{h}\subset |C_{h}|\times X\] 
be the universal curve.

When $E_{0}\cong \oo_{X}$, by \cite[Proposition B.8]{pandharipande-thomasiii}, 
\begin{equation*}
S_{X}^{\oo_{X}}(\beta_{h},1-h+d)\cong \mathcal C_{h}^{[d]}\end{equation*}
where $\mathcal C_{h}^{[d]}$ is the relative Hilbert scheme of points. 
If there is an ample line bundle $H$ such that
\begin{equation}\label{minimal-degree}C_h.H=\min\{L.H|L\in {\rm Pic}(X),\ L.H>0\},\end{equation}
then $S_{X}^{\oo_{X}}(\beta_{h},1-h+d)$ is a smooth scheme of dimension $h+d$, 
see \cite[Lemma 5.117, Lemma 5.175]{kawai-yoshioka} or \cite[Proposition C.2]{pandharipande-thomasiii}. 

The moduli space is not smooth in general for a higher rank $E_{0}$. For example, assume $E_{0}\cong \oo_{X}^{\oplus 2}$ and the stable pair $(E,\alpha:\oo_{X}^{\oplus 2}\to E)$ maps a summand $\oo_{X}$ to $0$. Then, the deformation space of this stable pair is 
\[\Hom(I^{\bullet},E)\cong \Hom(\oo_{X}\to E, E)\oplus H^{0}(E).\]
The dimension of $\Hom(\oo_{X}\to E, E)$ is $h+d$, while $h^{0}(E)$ may vary as $E$ varies. But when $d$ is large, we do expect the moduli space to be smooth for higher rank $E_{0}$.

\begin{prop}\label{higher-rank-smooth}Suppose $\beta_{h}$ is irreducible, i.e. $\beta_{h}$ is not a sum of two curve classes, and $d> 2h-2$. Then the moduli space $S_{X}^{\oo_{X}^{\oplus r}}(\beta_{h}, 1-h+d)$ is smooth of dimension $rd+(r-2)(1-h)+1$.
\end{prop}

\begin{proof}Apply the functor $\Hom(-,E)$ to
\[I^{\bullet}\to \oo_{X}^{\oplus r}\to E\to I^{\bullet}[1].\]
According to Lemma \ref{no-higher-obstructions}, the associated long exact sequence is 
\begin{table*}[h]
\centering
\begin{tabular}{cccccccc c}
0 & $\to$ & $\Hom(E,E)$ & $\to$  & $H^{0}(X, E)^{\oplus r}$ & $\to$ & $\Hom(I^{\bullet},E)$ & $\to$ &\\
&&$\Ext^{1}(E,E)$ & $\to$  & $H^{1}(X, E)^{\oplus r}$ & $\to$ & $\Ext^{1}(I^{\bullet},E)$ & $\to$ & \\
&& $\Ext^{2}(E,E)$ & $\to$ & 0 &   &  &  & 
\end{tabular}
\end{table*}

Since $\beta_{h}$ is irreducible, $E$ is stable. Therefore, $\ext^{2}(E,E)=\hom(E,E)=1$. When $d>2h-2$, by Serre duality, $h^{1}(X,E)=h^{1}(C,E)=0$ where $C$ is the support of $E$.
Thus, the tangent space $\Hom(I^{\bullet},E)$ has constant dimension
$\chi(I^{\bullet},E)+1 
=rd+(r-2)(1-h)+1$.
\end{proof}

For every $h\geq 0$, there exists a K3 surface $X_{h}$ and a curve class $\beta_{h}\in H_{2}(X_{h},\Z)$, such that $\beta_{h}.\beta_{h}=2h-2$ and (\ref{minimal-degree}) is satisfied, see \cite[Remark 5.110]{kawai-yoshioka}. For each $h\geq 0$, we fix such $X_{h}$ and $\beta_{h}$.

Kawai and Yoshioka calculated the generating series of topological Euler characteristics of the moduli spaces \cite[Corollary 5.85]{kawai-yoshioka}.
\begin{customthm}{KY}[Kawai-Yoshioka]\label{kawai-yoshioka}For $0<|q|<|y|<1$, the generating series of topological Euler characteristics is
\begin{eqnarray*}
&&\sum_{h=0}^{\infty}\sum_{d=0}^{\infty}\chi_{\rm top}\big(S_{X_{h}}^{\oo}(\beta_{h},1-h+d)\big)q^{h-1}y^{1-h+d}\\
&=& \frac{1}{(y^{-1/2}-y^{1/2})^{2}q\prod_{n=1}^{\infty}(1-q^n)^{20}(1-q^ny)^2(1-q^ny^{-1})^2}.
\end{eqnarray*}
\end{customthm}

\vskip10pt
Next, we consider stable pairs over $X_{h}$ of the form
\[\alpha:L_{h}\to E,\]
where $L_{h}$ is a line bundle with the first Chern class $c_{1}(L_{h})=l\beta_{h}$. Such a stable pair is equivalent to 
\[\oo_{X}\to E\otimes L^{-1}_{h}.\]
Notice that $c_{1}(E\otimes L^{-1}_{h})=\beta_{h}$ and $\chi(E\otimes L^{-1}_{h})=1-h+d-2l(h-1)$.
Therefore, 
\[S^{L_{h}}_{X_{h}}(\beta_{h},1-h+d)\cong S^{\oo_{X}}_{X_{h}}(\beta_{h},1-h+d-2l(h-1)).\]
If $\alpha\not= 0$, then $d\geq 2l(h-1)$. The generating series is
\begin{eqnarray*}
&& \sum_{h=0}^{\infty}\sum_{d=2l(h-1)}^{\infty}\chi_{\rm top}
\big(
S^{L_{h}}_{X_{h}}(\beta_{h},1-h+d)
\big)
q^{h-1}y^{d+1-h}\\
&=&\sum_{h=0}^{\infty}\sum_{d=0}^{\infty}\chi_{\rm top}
\big(
S^{\oo_{X}}_{X_{h}}(\beta_{h},1-h+d)
\big)
(qy^{2l})^{h-1}y^{d+1-h}\\
&=&\frac{1}{(y^{-1/2}-y^{1/2})^{2}qy^{2l}\prod_{n=1}^{\infty}(1-q^ny^{2nl})^{20}(1-q^ny^{2nl+1})^2(1-q^ny^{2nl-1})^2}
\end{eqnarray*}

\vskip10pt
Now, we consider stable pairs over $X_{h}$ of the form
\[\alpha:\bigoplus_{i} L_{i,h}\to E,\]
where $L_{i,h}$ is a line bundle with $c_{1}(L_{i,h})=l_{i}\beta_{h}$. 
The proof of Proposition~\ref{higher-rank-smooth} can also show that the moduli space  is smooth when $d$ is large compared to $l_{i}$ and $h$. 
Let $\mathbb G_{m}$ act on direct summands with distinct weights, then there is a natural $\mathbb G_{m}$-action on the moduli space $S^{\oplus L_{i,h}}_{X_{h}}(\beta_{h},1-h+d)$. A morphism $\oplus L_{i,h}\to E$ is fixed under the action if and only if exactly one summand $L_{i,h}$ is mapped to $E$ nontrivially. Thus, the fixed loci
\begin{equation*}
S^{\oplus L_{i,h}}_{X_{h}}(\beta_{h},1-h+d)^{\mathbb G_{m}}\cong \coprod_{i} S^{ L_{i,h}}_{X_{h}}(\beta_{h},1-h+d).
\end{equation*}
When $\alpha\not=0$, $d\geq \min\{2l_{i}(h-1)\}$. To calculate the Euler characteristics, we can use the localization formula, even when the moduli space is not smooth \cite{lawson-yau}. Then,
\begin{eqnarray*}
&& \sum_{h}\sum_{d}\chi_{\rm top}
\big(
S^{\oplus L_{i,h}}_{X_{h}}(\beta_{h},1-h+d)
\big)
q^{h-1}y^{d+1-h}\\
&=&\sum_{i}\frac{1}{(y^{-1/2}-y^{1/2})^{2}qy^{2l_{i}}\prod_{n=1}^{\infty}(1-q^ny^{2nl_{i}})^{20}(1-q^ny^{2nl_{i}+1})^2(1-q^ny^{2nl_{i}-1})^2}.
\end{eqnarray*}
 


\vskip45pt
\appendix
\section*{Appendix. The case $\deg \delta<\deg P$}\label{small-delta}
\vskip35pt

This appendix contains the proofs of parallel statements when $\deg \delta<\deg P$ and a remark on critical values of $\delta$.

\begin{proof}[Proof of Lemma~\ref{mu-min}]
Take the Harder-Narasimhan filtration $\{F_{t}\}_{1\leq t\leq l}$ of $E$ with respect to slope. If the induced map 
\begin{equation*}E_{0}\to F_{l}/F_{l-1}=:\gr_{l}E\end{equation*} is nonzero, then the argument for the case where $\deg\delta\geq d$ works here. 
Otherwise, $\im\alpha\subset F_{l-1}$. Thus, by semistability, 
\[\frac{\delta}{r(F_{l-1})}+p_{F_{l-1}}\leq \frac{\delta}{r(E)}+p_{E}. \]
Therefore, $p_{F_{l-1}}\leq p_{E}$, which in turn implies that $p_{\gr_{l}E}\geq p_{E}$. Thus, 
\[\mu_{\min}(E)=\mu(\gr_{l}E)\geq \mu(E).\]
Therefore, $\mu_{\min}(E)$ is bounded below by a constant determined by $P$ and $X$.
\end{proof}

In the construction of the moduli space, we need to replace Lemma~\ref{EquiSS2} by the following lemma.
\begin{lem}\label{EquiSS1}Fix $P$ and $\delta$ with $\deg\delta< \deg P$. Then there is an $m_{0}\in \Z_{>0}$, such that for any integer $m\geq m_{0}$ and any pair $(E,\alpha)$, where $E$ is a pure with $P_{E}=P$ and multiplicity $r(E)=r$, the following assertions are equivalent.
\begin{enumerate}[i)]
\item The pair $(E,\alpha)$ is stable.
\item $P_{E}(m)\leq h^{0}(E(m))$, and for any nontrivial proper sub-pair $(G,\alpha^{\prime})$ with $G$ of multiplicity $r(G)$,
\[\frac{h^{0}((G,\alpha^{\prime})(m))}{r(G)}< p_{(E,\alpha)}(m).\]
\item For any proper quotient pair $(F,\alpha^{\prime\prime})$ with $F$ of dimension $d$ and multiplicity $r(F)$,
\[\frac{h^{0}((F,\alpha^{\pprime})(m))}{r(F)}> p_{(E,\alpha)}(m).\]
\end{enumerate}
Here, 
\[h^{0}((G,\alpha^{\prime})(m)) =h^0(G(m))+\epsilon(\alpha^{\prime})\delta(m),\] 
and $h^{0}((F,\alpha^{\pprime})(m))$ has a similar meaning.
\end{lem}
\begin{proof}
A large part of the proof is the same as that of Lemma~\ref{EquiSS2}. Again, the proof will proceed as follows: $i)\Rightarrow ii)\Rightarrow iii)\Rightarrow i)$.
\vskip6pt
i) $\Rightarrow$ ii):
With the same notation, we have the inequalities (\ref{1to2BySimpson}) and (\ref{1to2}). 
Therefore, when $\nu\leq\nu_{0}$, we can further enlarge $m_1$ such that $\forall m\geq m_1$\[\frac{1}{d!}\big((1-\frac{1}{r})([\mu+m+B]_{+})^{d}+\frac{1}{r}([\nu+m+B]_{+})^{d}\big)+\frac{\epsilon(\alpha^{\prime})\cdot\delta(m)}{r(G)}<\frac{P(m)}{r}+\frac{\epsilon(\alpha)\delta(m)}{r}.\]
This is because there are only finitely many choices for $\epsilon(\alpha^{\prime})/r(G)$. Hence, for $m\geq m_1$ and $\nu\leq\nu_{0}$,
\[\frac{h^{0}((G,\alpha^{\prime})(m))}{r(G)}<p_{(E,\alpha)}(m).\]

When $\nu>\nu_{0}$, by the same argument as in the proof of Lemma~\ref{EquiSS2}, we can enlarge $m_{0}$ again, if necessary, such that for $m\geq m_{0}$, 
$P_{G}(m)=h^{0}(G(m))$
 and
 \[p_{(G,\alpha^{\prime})}< p_{(E,\alpha)} \iff p_{(G,\alpha^{\prime})}(m)< p_{(E,\alpha)}(m).\]
Therefore, either $\nu\leq\nu_{0}$ or $\nu>\nu_{0}$,
\[\frac{h^0((G,\alpha^{\prime})(m))}{r(G)}< p_{(E,\alpha)}(m).\]
\vskip6pt
ii) $\Rightarrow$ iii): Also by studying the exact sequence (\ref{2to3}),
\begin{eqnarray*}\frac{h^{0}(F(m))+\epsilon(\alpha^{\pprime})\cdot\delta(m)}{r(F)}> \frac{h^{0}(E(m))+\epsilon(\alpha)\cdot\delta(m)}{r}\geq p_{(E,\alpha)}(m).\end{eqnarray*}

\vskip6pt
iii) $\Rightarrow$ i): We also have the inequality (\ref{3to1}).
Denote by $(K,\beta)$ the induced quotient pair. (Here, we change the notation.) By the hypothesis and (\ref{3to1}),
\begin{equation*}
\frac{P(m)+\epsilon(\alpha)\delta(m)}{r(E)} < \frac{1}{d!}([\mu(K)+m+C]_{+})^{d}+\frac{\epsilon(\beta)\delta(m)}{r(K)},
\end{equation*}
For large $m$, the right hand becomes an polynomial in $m$. The leading coefficients of both sides are the same. By considering the second coefficients, we deduce that $\mu_{\min}(E)$ is bounded below. Thus, $\mu_{\max(E)}$ is bounded above. Therefore, by Theorem~\ref{BddMuMax}, the family of pure sheaves $E$ satisfying the hypothesis for large $m$ is bounded.

If $(E,\alpha)$ is not stable, we denote by $(F,\alpha^{\pprime})$ the last Harder-Narasimhan factor of such a pair $(E,\alpha)$, which is a proper quotient. Then
\[p_{(F,\alpha^{\pprime})}\leq p_{(E,\alpha)}.\] 
By Theorem~\ref{GrSlope}, the family of these $F$'s is bounded. Thus, enlarge $m_{0}$ if necessary,
 \[h^0(F(m))=P_F(m)\quad \mbox{and}\quad h^0(E(m))=P(m),\quad\forall m\geq m_0.\] 
 Then
 \[\frac{h^0(F(m))+\epsilon(\alpha^{\pprime})\delta(m)}{r(F)}\leq p_{(E,\alpha)}(m),\]
 contradicting the hypothesis. So, $(E,\alpha)$ is stable.
 \end{proof}
Semistability can be characterized similarly, replacing the two strong inequalities by  weak inequalities.
\vskip8pt
In defining the $\SL(V)$-linearized line bundle $L$, let 
\[\frac{n_{1}}{n_{2}}=P(l)\frac{\delta(m)}{P(m)+\delta(m)}-\delta(l)\frac{P(m)}{P(m)+\delta(m)}.\]
This number is positive for $l$ large enough, due to the degree condition.
\begin{proof}[Proof of Proposition~\ref{gitss-semistable}]
Suppose $[\sigma]\times [q]$ is GIT-semistable. 
With the same notations as in the case where $\deg\delta\geq d$ and the new assignment of $n_{1}/n_{2}$,
\begin{eqnarray*}
r(E)-r''\geq r(G)\geq\frac{\dim W}{\dim V}\cdot\frac{r(E)P(m)}{P(m)+\delta(m)}+\epsilon_W(\sigma)\frac{r(E)\delta(m)}{P(m)+\delta(m)}.
\end{eqnarray*}Recall that $P(m)=\dim V$. Therefore,
\begin{eqnarray*}
\frac{\dim V-\dim W+(1-\epsilon_{W}(\sigma))\delta(m)}{r''}\geq\frac{P(m)+\epsilon(\alpha)\delta(m)}{r(E)}.
\end{eqnarray*}
By the same argument as before, if $\alpha^{\pprime}=\pi\circ\phi\circ\alpha=0$, then $\im\sigma\subset W$. From the last inequality and (\ref{GEF''}), we get
\[\frac{h^0(F''(m))+\epsilon(\alpha^{\pprime})\delta(m)}{r''}\geq \frac{P(m)+\epsilon(\alpha)\delta(m)}{r(E)}.\]
By the same argument, replacing Lemma~\ref{EquiSS2} by Lemma~\ref{EquiSS1}, we deduce that $(E,\alpha)$ is semistable.

Next, we assume that $(E,\alpha)$ is stable, and $q(m)$ induces an isomorphism between global sections. For any subspace $0\not= W< V$, let $G=q( W\otimes\oo(-m))$ and $(G,\alpha^{\prime})$ the corresponding sub-pair. If $(G,\alpha^{\prime})=(E,\alpha)$, the inequality (\ref{WG-polynomial}) holds. Assume that $(G,\alpha^{\prime})$ is a proper sub-pair. According to Lemma~\ref{EquiSS1}, we have
\[\frac{h^{0}(G(m))+\epsilon(\alpha^{\prime})\delta(m)}{r(G)}< \frac{h^{0}(E(m)+\epsilon(\alpha)\delta(m))}{r},\]
which gives (the inverses of) the coefficients of two sides in the inequality (\ref{WG-polynomial}). Thus, $([\sigma],[q])$ is stable in the GIT sense.

What is left to be proven is that, if $(E,\alpha)$ is a strictly semistable pair, then a point $[\sigma]\times [q]$ in $\bar Z$ with associated $(E,\alpha)$ is strictly semistable in the GIT sense. Suppose $(G,\alpha^{\prime})$ is a destabilizing sub-pair, let 
\[W=H^0(G(m))\subset H^0(E(m))\cong V.\]
Then
\[\epsilon(\alpha^{\prime})=\epsilon_{W}(\sigma).\]
 It is enough to show that
\[P_G=\Big(P(l)\frac{\delta(m)}{P(m)+\delta(m)}-\delta(l)\frac{P(m)}{P(m)+\delta(m)}\Big)
\Big(\epsilon_{W}(\sigma)-\frac{\dim W}{\dim V}\Big)
+P(l)\frac{\dim W}{\dim V}.\]
By our choice of $m$, the right hand side equals
\[(P+\delta)\cdot\frac{\dim W+\epsilon_W(\sigma)\delta}{P(m)+\delta(m)}-\epsilon_W(\sigma)\delta=(P+\delta)\frac{r(G)}{r(E)}-\epsilon(\alpha^{\prime})\delta=P_{G}.\]

Thus, we have finished proving the cases where $\deg \delta<d$.
\end{proof}

\vskip8pt
\noindent{\it Remark.} Fix the smooth projective variety $(X,\oo_{X}(1))$, coherent sheaf $E_{0}$ and the Hillbert polynomial $P$. When stability condition $\delta$ varies, the moduli space $S_{X}(P,\delta)$ also undergoes some changes. 

A $\delta$ is {\it regular} if there are two polynomials $\delta_{1}$ and $\delta_{2}$, such that $0<\delta_{1}<\delta<\delta_{2} $ and for any $\delta_{0}\in(\delta_{1},\delta_{2})$, the set of $\delta_{0}$-semistable pairs of type $P$ remains constant. Otherwise, $\delta$ is called {\it critical}. We have the following statement similar to \cite[Theorem 4.2]{MR1644040}:
\begin{prop}Fix $(X,\oo_{X}(1))$, $E_{0}$ and $P$. There are only finitely many critical values and they are all of degree $<\deg P$.
\end{prop}
\begin{proof}Suppose $\delta $ is critical. Let $(E,\alpha)$ be a strictly $\delta$-semistable pair. Then there is a proper sub-pair $(E^{\prime},\alpha^{\prime})$ such that
\[p_{E}+\frac{\delta}{r(E)}=p_{E^{\prime}}+\frac{\epsilon(\alpha^{\prime})\delta}{r(E^{\prime})}.\]
Therefore, each critical $\delta$ has the following form
\begin{equation}\label{critical-delta}\delta= \frac{r(E^{\prime})P_{E}-r(E)P_{E^{\prime}}}{\epsilon(\alpha^{\prime})r(E)-r(E^{\prime})}.\end{equation}

If $\alpha^{\prime}=0$, then $p_{E^{\prime}}> p_{E}$. Thus, $\mu(E^{\prime})$ is bounded below by the constant $ \mu(E)$ determined by $P$.
If $\alpha^{\pprime}\not=0$, then there is a nonzero map $\alpha^{\prime}:E_{0}\to E^{\prime}$. The proof of Lemma~\ref{mu-min} also shows that $\mu(E^{\prime})$ is bounded below by a constant determined by $X$ and $P$. Notice that $E/E^{\prime}$ is pure, since $(E,\alpha)$ is semistable. Therefore, such destabilizing sub-pairs $(E^{\prime},\alpha^{\prime})$ form a bounded family, according to Theorem~\ref{GrSlope}. So, there are only finitely many such $P_{E^{\prime}}$'s. Hence, there are only finitely many $\delta$'s of the form (\ref{critical-delta}).
\end{proof}

\medskip

\vskip20pt

\quad\quad\quad\textsc{Department of Mathematics, Northeastern University}
\vskip6pt
\indent\quad\quad\quad{\it Email:} \href{mailto:lin.yinb@husky.neu.edu}{\texttt{lin.yinb@husky.neu.edu}}
\end{document}